\documentclass[12pt,leqno,a4paper]{amsart}
\usepackage{amssymb,enumerate}
\usepackage{xy}
\xyoption{all}

\newcommand{\CC}{{\mathbb{C}}}
\newcommand{\FF}{{\mathbb{F}}}
\newcommand{\QQ}{{\mathbb{Q}}}
\newcommand{\ZZ}{{\mathbb{Z}}}

\newcommand{\fA}{{\mathfrak{A}}}
\newcommand{\fS}{{\mathfrak{S}}}
\newcommand{\fp}{{\mathfrak{p}}}
\newcommand{\fsl}{{\mathfrak{sl}}}

\newcommand{\bB}{{\mathbf{B}}}
\newcommand{\bG}{{\mathbf{G}}}
\newcommand{\bb}{{\mathbf{b}}}

\newcommand{\cD}{{\mathcal{D}}}
\newcommand{\cE}{{\mathcal{E}}}
\newcommand{\cO}{{\mathcal{O}}}

\newcommand{\Ph}[1]{\Phi_{#1}}

\newcommand{\Irr}{{\operatorname{Irr}}}
\newcommand{\rk}{{\operatorname{rk}}}
\newcommand\tfA{{\tilde\fA}}
\newcommand\tfS{{\tilde\fS}}

\newcommand{\End}{{\operatorname{End}}}
\newcommand{\GL}{{\operatorname{GL}}}
\newcommand{\PGL}{{\operatorname{PGL}}}
\newcommand{\PSL}{{\operatorname{L}}}
\newcommand{\SL}{{\operatorname{SL}}}
\newcommand{\Sp}{{\operatorname{Sp}}}
\newcommand{\PSp}{{\operatorname{S}}}
\newcommand{\Spin}{{\operatorname{Spin}}}
\newcommand{\OO}{{\operatorname{O}}}
\newcommand{\SU}{{\operatorname{SU}}}

\DeclareMathOperator{\Res}{Res}    
\DeclareMathOperator{\Ind}{Ind}    
\DeclareMathOperator{\lse}{se}    

\newcommand{\tV}{{\tilde V}}

\newcommand{\tw}[1]{{}^#1\!}

\newcommand{\Chevie}{{\sf Chevie}}
\newcommand{\MAGMA}{{\sf Magma}}

\let\lang=\langle
\let\rang=\rangle
\let\al=\alpha
\let\la=\lambda
\let\lra=\longrightarrow

\newtheorem{thm}{Theorem}[section]
\newtheorem{lem}[thm]{Lemma}
\newtheorem{conj}[thm]{Conjecture}
\newtheorem{cor}[thm]{Corollary}
\newtheorem{prop}[thm]{Proposition}

\theoremstyle{definition}

\theoremstyle{remark}
\newtheorem{rem}[thm]{Remark}

\begin{document}

\title[Simple endotrivial modules]{Simple endotrivial modules\\ for quasi-simple groups}

\date{\today}

\author{Caroline Lassueur, Gunter Malle and Elisabeth Schulte}
\address{FB Mathematik, TU Kaiserslautern, Postfach 3049,
         67653 Kaisers\-lautern, Germany.}
\email{lassueur@mathematik.uni-kl.de}
\email{malle@mathematik.uni-kl.de}
\email{eschulte@mathematik.uni-kl.de}

\thanks{The authors gratefully acknowledge financial support by ERC
  Advanced Grant 291512.}



\dedicatory{Dedicated to Geoffrey Robinson on the occasion of his 60th birthday}

\begin{abstract}
We investigate simple endotrivial modules of finite quasi-simple groups and
classify them in several important cases. This is motivated by a recent result
of Robinson \cite{Rob} showing that simple endotrivial modules of most groups
come from quasi-simple groups.
\end{abstract}

\maketitle


\section{Introduction} \label{sec:intro}

Let $G$ be a finite group and $k$ a field of prime characteristic $p$ such that
$p$ divides $|G|$. A $kG$-module $V$ is called \emph{endotrivial} if
$V\otimes V^*\cong k\oplus P$, with a projective $kG$-module $P$. The tensor
product of $kG$-modules induces a group structure on the set of isomorphism
classes of indecomposable endotrivial $kG$-modules, called the \emph{group of
endotrivial modules} and denoted $T(G)$. Endotrivial modules have seen a
considerable interest in the last fifteen years, eventually leading to the
determination of the Dade group for all $p$-groups (see, for example,
\cite{Th} and the references therein), and of $T(G)$ for some
classes of general groups (see \cite{C12,CHM10,CMN06,CMN09,MaTh}).

A recent paper of Robinson \cite{Rob} has put the focus on simple endotrivial
modules for quasi-simple groups. He shows that whenever the Sylow $p$-subgroups
of a finite group $G$ are neither cyclic nor quaternion, then any simple
endotrivial $kG$-module is either simple endotrivial for a quasi-simple normal
subgroup, or induced from a 1-dimensional module of a strongly $p$-embedded
subgroup of $G$.
Note that, in contrast, for $p$-solvable groups of $p$-rank $>1$ all simple
endotrivial modules are 1-dimensional by Navarro--Robinson \cite{NR12}.

The purpose of the present paper is to start a classification of simple
endotrivial modules for quasi-simple groups. We obtain a complete description
of such modules in several important cases. Our first main result is a precise
condition for the existence of faithful simple endotrivial modules in the
case of cyclic Sylow $p$-subgroups, depending on their location 
on the Brauer tree (see Theorem~\ref{thm:ETleaves}), and a
complete classification for all such cases in quasi-simple groups not of
classical Lie type or of rank at least~4, up to a small number of cases in
sporadic groups in which the information on Brauer trees is yet incomplete.
\par
Our second main result is the complete classification of simple endotrivial
modules for covering groups of alternating groups, see
Theorem~\ref{thm:mainAn}.
\par
The third main result concerns groups of Lie type: for $p$ the defining
characteristic we obtain a complete classification of all simple endotrivial
modules in Theorem~\ref{thm:defchar}; for non-defining characteristic we
determine all simple endotrivial modules of exceptional groups of
rank less than four and their covering groups.

Our results also indicate the validity of stronger statements: the
existence of faithful simple endotrivial modules seems to drastically restrict
the structure of Sylow $p$-subgroups:

\begin{conj}   \label{conj:gen}
 Let $G$ be a finite quasi-simple group with a faithful simple endotrivial
 module. Then the Sylow $p$-subgroups of $G$ have rank at most~2.
\end{conj}

Here, the rank of a $p$-group $H$ is the maximal rank of an elementary abelian
subgroup of $H$. In fact, in all known examples, the Sylow $p$-subgroups are
either homocyclic of rank at most~2, extraspecial of order~$p^3$ with
$p\le 11$, or dihedral (in $\PSL_2(q)$ with $q\equiv-1\pmod4$, see
Proposition~\ref{prop:SL2}).
As a consequence of our classifications we obtain:

\begin{thm}   \label{thm:gen}
 Conjecture~\ref{conj:gen} holds in all of the following cases:
 \begin{enumerate}
  \item[\rm(a)] if $p=2$,
  \item[\rm(b)] if $G/Z(G)$ is an alternating group,
  \item[\rm(c)] if $G/Z(G)$ is a sporadic group,
  \item[\rm(d)] if $G/Z(G)$ is a group of Lie type and $p$ is its defining
   characteristic, and
  \item[\rm(e)] if $G/Z(G)$ is an exceptional group of Lie type.
 \end{enumerate}
\end{thm}

Part~(a) of this claim is shown in Theorem~\ref{thm:l=2}, part~(b) follows from
Theorem~\ref{thm:mainAn}, part~(c) from Theorem~\ref{thm:mainspor}, part~(d)
from Theorem~\ref{thm:defchar}, and part~(e) is Theorem~\ref{thm:exc1}.
An a priori proof of Conjecture~\ref{conj:gen} would considerably simplify
the classification of simple endotrivial modules.
 
The new feature of our approach is that the proofs rely mainly on character
theoretic methods. This is made possible through our generalization of a
lifting result due to Alperin \cite{Al01} from the case of $p$-groups to
arbitrary finite groups, which may be of independent interest:

\begin{thm}   \label{thm:lift}
 Let $(K,\cO,k)$ be a $p$-modular system and $V$ be an endotrivial $kG$-module.
 Then $V$ lifts to an endotrivial $\cO G$-module. 
\end{thm}

Previous work on endotrivial modules for different classes of quasi-simple
groups $G$ was concerned with the determination of the group $T(G)$ of
endotrivial modules. This includes results on groups with cyclic Sylow
subgroups \cite{MaTh}, the symmetric and alternating groups \cite{CHM10,CMN09},
and groups of Lie type in their defining characteristic \cite{CMN06}. 
Nevertheless there are three main obstructions to use these results to answer
the current question of finding the simple endotrivial modules. First, the
aforementioned articles do not treat covering groups. Second, they determine
the structure of $T(G)$ but not the indecomposable endotrivial modules
themselves, in that their description involves Green correspondence, which
is not explicit and notoriously difficult to determine. Third, even in the
simplest cases where $T(G)=\left<\Omega(k)\right>\cong\ZZ$ (where $\Omega$ is the Heller operator), it is not
clear whether any of the modules $\Omega^n(k)$ for $n\in\ZZ$ can be simple.
Finally, even in those few cases where the cited papers obtain explicit
descriptions of some simple endotrivial modules, our approach seems more
straightforward than the more intricate module theoretic methods.

Finally, a comparison of our results with those of
\cite{C12,CHM10,CMN06,CMN09,MaTh} seems to indicate a strong link between the
order of the
torsion subgroup of $T(G)$ and the number of simple endotrivial modules.
Therefore a classification of simple endotrivial modules may appear to be
an important step towards the final description of the group $T(G)$.
Our study of sporadic groups via character-theoretic methods also provides us
with new torsion elements of $T(G)$ unknown in the literature so far, see
Remark~\ref{rem:tors}.
\medskip

The paper is built up as follows: in Section~\ref{sec:pre} we collect some
basic facts and prove Theorem~\ref{thm:lift}. In Section~\ref{sec:cycl}
we obtain rather strong results in the case of cyclic Sylow subgroups.
In Section~\ref{sec:alt} we classify the simple endotrivial modules for
covering groups of alternating groups, and in Section~\ref{sec:defchar} those
for groups of Lie type in their defining characteristic. In
Section~\ref{sec:cross} we obtain partial results for simple groups of Lie
type, mostly of exceptional type, in cross characteristic, and in the final
section we classify simple endotrivial modules for sporadic groups and their
covering groups, up to a few open cases of very large dimension.

\vskip 1pc\noindent
{\bf Acknowledgement:} We thank the anonymous referee for her/his numerous
detailed remarks and comments which led to an improvement of our paper.

\section{Preliminaries} \label{sec:pre}

Throughout, unless otherwise stated, we assume $G$ is a finite group and $k$
an algebraically closed field of prime characteristic $p$ such that $p$
divides $|G|$. We let $(K,\cO,k)$ be a splitting $p$-modular system, and
let $\fp:=J(\cO)$. For a block $\bB$ of $kG$ we write $\lse(\bB)$ for the
number of isomorphism classes of simple endotrivial $\bB$-modules. For
background material on endotrivial modules we refer to
\cite{C12,CHM10,CMN06,CMN09,MaTh}. We first collect some elementary facts.

\begin{lem}   \label{lem:deg}
 Let $V$ be an endotrivial $kG$-module, with $k$ a field of
 characteristic~$p$. Then $\dim V\equiv\pm1\pmod{|G|_p}$ for $p\ge3$,
 respectively $\dim V\equiv\pm1\pmod{\frac{1}{2}|G|_2}$ when $p=2$.
\end{lem}

\begin{proof}
Since $V$ is endotrivial, $V\otimes V^*\cong k\oplus\text{(projective)}$, so
$(\dim V)^2-1\equiv0\pmod{|G|_p}$. When $p>2$, this means that either
$\dim V+1$ or $\dim V-1$ is divisible by $|G|_p$, while for $p=2$, at least
one of the factors is divisible by $\frac{1}{2}|G|_2$.
\end{proof}

\begin{lem}   \label{lem:sub}
 Let $V$ be a $kG$-module and $H\le G$ containing a Sylow $p$-subgroup of $G$.
 Then $V$ is endotrivial if and only if $V|_H$ is endotrivial.
\end{lem}

\begin{proof}
The claim follows immediately from the fact that a $kG$-module is projective
if and only if its restriction to a Sylow $p$-subgroup of $G$ is projective.
\end{proof}

\begin{cor}   \label{cor:zero}
 Let $V$ be a $kG$-module which is liftable to a simple $\CC G$-module
 with character $\chi$, say. If $V$ is endotrivial, then $|\chi(g)|=1$ for
 all $p$-singular elements $g\in G$.
\end{cor}

\begin{proof}
By assumption $\chi\bar\chi\equiv 1+\psi\pmod p$, where $\psi$ is the
character of a $p$-projective module. Thus, $\psi(g)=0$ for all $p$-singular
elements $g\in G$. 
The claim follows.
\end{proof}

We now prove Theorem~\ref{thm:lift} from the introduction. This result
is due to Alperin for $p$-groups \cite{Al01} (for which the image of any
representation lies in the special linear group). His proof generalises easily.
We sum up here the main ideas, based on a detailed exposition of the proof
for $p$-groups written in \cite{Urf}.

\begin{prop}   \label{prop:lift}
 Let $V$ be an endotrivial $kG$-module such that the image of the corresponding
 representation $\rho: G\rightarrow \GL_{n}(k)$ lies in $\SL_{n}(k)$. Then $V$
 lifts to an endotrivial  $\cO G$-module.
\end{prop}

\begin{proof}
Denote by $\SL_{n}(\cO,m)$ the congruence subgroups of $\SL_{n}(\cO)$, i.e.,  
the set of determinant 1 matrices congruent to the identity matrix $I_{n}$
modulo $\fp^{m}$. They form a central series of $\SL_{n}(\cO)$ with
successive quotients all isomorphic to $\fsl_{n}(k)$ (the $n\times n$-matrices
of trace zero) as $\SL_{n}(k)$-modules. Since $V$ is endotrivial,
$\End_{k}(V)\cong k\oplus\text{(projective)}$, and $n\equiv\pm 1\pmod{p}$ by
Lemma~\ref{lem:deg}, thus we also have $\End_{k}(V)\cong k\oplus U$ where $U$
is the kernel of the trace map on $\End_{k}(V)$. Hence $\fsl_{n}(k)$, seen
as a $kG$-module via~$\rho$, must be projective.\par
Taking a pull-back $X_{2}$ of $\rho$ and the homomorphism induced by reduction
modulo~$\fp$ from $\SL_{n}(\cO)/\SL_{n}(\cO,2)\rightarrow \SL_{n}(k)$, which
has kernel $\SL_{n}(\cO,1)/\SL_{n}(\cO,2)\cong \fsl_{n}(k)$, yields a group
extension $1\rightarrow\fsl_{n}(k)\rightarrow X_{2}\rightarrow G\rightarrow 1$.
This extension splits because $\fsl_{n}(k)$ is a projective (= injective)
$kG$-module. As a consequence, $\rho$ lifts to a homomorphism
${\rho_{2}: G\rightarrow \SL_{n}(\cO)/\SL_{n}(\cO,2)}$. Inductively, for every
$m>2$, one thus constructs a homomorphism
$\rho_{m}:G\rightarrow \SL_{n}(\cO)/\SL_{n}(\cO,m)$ lifting
$\rho_{m-1}:G\rightarrow \SL_{n}(\cO)/\SL_{n}(\cO,m-1)$.  Finally
$\SL_{n}(\cO)\cong \lim\limits_{\longleftarrow}{}_{m\geq2}\SL_{n}(\cO)/\SL_{n}(\cO,m)$,
so that the universal property of the projective limit yields the desired
group homomorphism $\tilde\rho:G\longrightarrow \SL_{n}(\cO)$ lifting $\rho$.
\par
Moreover if $M$ is an $\cO G$-module lifting $V$, then
$\rk_\cO(M)=\dim_k(V)\equiv\pm 1\pmod p$ by Lemma~\ref{lem:deg}, so that
$\End_\cO(M)\cong \cO\oplus N$ with $N$ the kernel of the trace map on
$\End_\cO(M)$. Reducing modulo $\fp$ yields $\End_k(V)\cong k\oplus N/\fp N$,
where $N/\fp N$ is projective by assumption. It follows that $N$ is
projective, see e.g.~\cite[\S 1 and \S 27]{ThBook}. Hence $M$ is endotrivial.
\end{proof}

\begin{proof}[Proof of Theorem~\ref{thm:lift}]
By passing to the image of the representation and choosing a basis of $V$ we
may assume that $G\le \GL_n(k)$. Let $G_1:=GC$ and $G_0:=G_1\cap\SL_n(k)$, with
$$C:=\{aI_n\mid a^n=\det(g) \text{ for some } g\in G \}.$$
Then $G_1$ is a central product of $G$ with $C$, and of $G_0$ and $C$.
As $|G_1:G|$ and $|G_1:G_0|$ are prime to $p$, the embedding $G_0\le\SL_n(k)$
is also endotrivial. Thus by Proposition~\ref{prop:lift} it lifts to an
endotrivial $\cO G_0$-module. Denoting the corresponding representation by
$\psi$ we thus have $\psi(G_0)\le\SL_n(\cO)$.
\par
Reduction modulo $\fp$ induces a bijection between the group of $p'$-roots
of unity in $\cO$ and roots of unity in $k$, sending
$\psi(G_0)\cap Z(\SL_n(\cO))$ onto $G_0\cap Z(\SL_n(k))$. The inverse defines
a lift of $C$ into $\{aI_n\mid a\in\cO^\times\}\le\GL_n(\cO)$, which agrees
with $\psi$ on $G_0\cap Z(\SL_n(k))$ and which we also denote by $\psi$. Then
$G_1=G_0C\cong \psi(G_0)\psi(C)\le\GL_n(\cO)$ is a faithful representation of
$G_1$ which lifts $G_1\le\GL_n(k)$. Again, as $|G_1:G|$ and $|G_1:G_0|$ are
prime to $p$, this gives an endotrivial $\cO G$-module lifting the initial
representation of $G$.
\end{proof}

\section{Groups with cyclic Sylow $p$-subgroups} \label{sec:cycl}

In this section $k$ denotes an algebraically closed field of characteristic
$p>0$, and $G$ denotes a finite group with a non-trivial cyclic Sylow
$p$-subgroup $P\cong C_{p^n}$ for some integer $n\geq 1$.

For any group $G$, denote by $T(G)$ its group of endotrivial modules and by
$X(G)$ the subgroup of $T(G)$ consisting of the one-dimensional $kG$-modules
(with group law induced by $\otimes_{k}$). This group can be identified with
the group of $k^{\times}$-valued linear characters of $G$ so that
$X(G)\cong (G/[G,G])_{p'}$, the $p'$-part of the abelianization of $G$.

Let $Z$ denote the unique subgroup of $P$ of order $p$ and let
$H:=N_{G}(Z)$. The structure of the group $T(G)$ is described in
\cite{MaTh} as follows: since $H$ is strongly $p$-embedded in $G$,
$T(G)\cong T(H)$ via restriction and inverse map induced by Green
correspondence. Furthermore, there is an exact sequence
$$0\lra X(H)\lra T(H)\overset{\Res^{H}_{P}}{\lra} T(P)\lra 0$$
so that
$$T(H)= \begin{cases}
     X(H) &  \text{ if } |P|=2; \\
    \langle X(H), [\Omega(k_{H})]\rangle  &   \text{ if } |P|\geq 3.
\end{cases}$$
In addition, $T(P)= \langle[\Omega(k_{P})]\rangle\cong \ZZ/2$ if $|P|\geq 3$
and $T(P)=\{0\}$ if $|P|=2$ by \cite{Da78II}.
Unless otherwise stated, we assume for the rest of the section that
$|P|\geq 3$, in which case  $|T(H):X(H)|=2$.

\subsection{Endotrivial modules and blocks of $kG$}

Indecomposable endotrivial modules have dimension prime to $p$, thus
have the Sylow subgroups as  vertices and lie in blocks with full defect. 
Henceforth $\bB$ denotes a block of $kG$ with defect group $P$, and $e_\bB$
denotes its inertial index, which corresponds to the number of simple 
modules in $\bB$. Moreover let $\bB_{0}$ denote the principal block of $kG$ and
set $e:=e_{\bB_{0}}=|N_{G}(Z):C_{G}(Z)|$, so that  $e\,|\,p-1$.

If $M$ is an indecomposable non-projective $kG$-module, denote by $f(M)$ its
$kH$-Green correspondent. If $L$ is an indecomposable non-projective
$kH$-module, denote by $g(L)$ its $kG$-Green correspondent. If $M$ belongs
to $\bB$, then $f(M)$ belongs to the Brauer correspondent $\bb$ of $\bB$.

The stable Auslander--Reiten quiver $\Gamma_{s}(\bB)$ of $\bB$ is a finite
tube $(\ZZ/e_\bB\ZZ)A_{p^n-1}$, so that any non-projective indecomposable
$kG$-module has $\Omega$-period $2e_\bB$. For background material and standard
notation and terminology on the Auslander-Reiten quiver, we refer the reader
to \cite[Chap.~4 \& Sec. 6.5]{Ben98}. We recall that Green correspondence sets
up an
$\Omega^{2}$-equivariant graph isomorphism between $\Gamma_{s}(\bB)$ and
$\Gamma_{s}(\bb)$. Furthermore the structure of the $kH$-Brauer correspondent
$\bb$ of a block $\bB$ of $kG$ is well-known: $\bb$ has $e_\bb =e_\bB$
simple modules, all of the same $k$-dimension. These $e_\bb$ simple modules
form one boundary $\Omega^{2}$-orbit of $\Gamma_{s}(\bb)$
\cite[Sec.~6.5]{Ben98}.
Nonetheless, the $kG$-Green correspondent of a simple $kH$-module is not
necessarily simple. Simple $\bB$-modules lie in the $e$ top and bottom
$\Omega^{2}$-orbits of $\Gamma_{s}(\bB)$, each of them on a different
diagonal \cite[Prop.~4.2]{B91a}.

\begin{lem}   \label{lem:2e}
 Let $\bB$ be a block of $kG$ containing an indecomposable endotrivial
 module $V$. Then:
 \begin{enumerate}
  \item[\rm(a)] $e_\bB=e$.
  \item[\rm(b)] $\bB$ contains $2e$ endotrivial modules. They are exactly
   the modules forming the two boundary $\Omega^{2}$-orbits of
   $\Gamma_{s}(\bB)$.
  \item[\rm(c)] The $\Omega^{2}$-orbit of $\Gamma_{s}(\bB)$ containing $V$
   consists of the modules $\Omega^{2n}(V)$, $1\leq n\leq e$, while the other
   boundary $\Omega^{2}$-orbit consists of the modules $\Omega^{2n-1}(V)$,
   $1\leq n\leq e$.
 \end{enumerate}
\end{lem}

\begin{proof}
The connected component $AR(V)$ of $V$ in the stable Auslander--Reiten quiver
of $kG$ is $\Gamma_{s}(\bB)$, a finite tube $(\ZZ/e_\bB\ZZ )A_{p^n-1}$.
Now by \cite[Thm. 2.3]{B91}, $AR(V)\cong AR(k)=\Gamma_{s}(\bB_{0})$ which
is a tube $(\ZZ/e\ZZ )A_{p^n-1}$. Hence $e_\bB=e$. By
\cite[Thm. 2.6]{B91}, the endotrivial modules in $\Gamma_{s}(\bB)$ are
exactly the modules lying on the two boundary $\Omega^{2}$-orbits. Whence
there are $2e$ of them. Part (c) is well-known (see e.g. \cite[\S 4]{B91a}).
\end{proof}

\begin{lem}   \label{lem:X/e}
 Let $\bB$ be a block of $kG$. Then:
 \begin{enumerate}
  \item[\rm(a)] $\bB$ contains an indecomposable endotrivial module if and
   only if its Brauer correspondent $\bb$ contains a one-dimensional
   $kH$-module.
  \item[\rm(b)] If $V$ is an indecomposable endotrivial $\bB$-module, then
   there exist $U\in X(H)$ such that either $V\cong g(U)$ or
   $V\cong g(\Omega(U))$.
  \item[\rm(c)] There are $|X(H)|/e$ blocks of $kG$ containing indecomposable
   endotrivial modules. Each of them contains at least one simple endotrivial
   module.
  \end{enumerate}
\end{lem}

\begin{proof}
By \cite[Thm. 3.6]{MaTh}, $T(G)\cong T(H)$ via Green correspondence on the
indecomposable endotrivial modules, i.e., an indecomposable $\bB$-module $V$
is endotrivial if and only if $f(V)$, lying in $\bb$, is endotrivial. Now,
the simple modules in $\bb$ are all of the same dimension and form one
boundary $\Omega^{2}$-orbit of $\Gamma_{s}(\bb)$. Thus parts~(b) and~(c) of
Lemma~\ref{lem:2e} together with the fact that $|T(H):X(H)|=2$, imply that
$T(H)$ contains exactly $|T(H)|/2$ simple modules, all of dimension one.
Whence (a). Part (b) then follows from parts~(b) and~(c) of Lemma~\ref{lem:2e}.
Finally again since $|T(G)|=|T(H)|=2|X(H)|$, parts~(a) and~(b) of
Lemma~\ref{lem:2e} force the number of blocks containing endotrivial modules
to be $|X(H)|/e$. Moreover each of them contains a simple endotrivial module
by~\cite[Thm. 3.7]{B91a}, which proves that there is at least one simple
$\bB$-module lying on an end $\Omega^{2}$-orbit of $\Gamma_{s}(\bB)$.
\end{proof}

\begin{cor}   \label{cor:bounds}
 The number of simple endotrivial modules over $kG$ is bounded below by
 $|X(H)|/e$ and bounded above by $|X(H)|=|H/[H,H]|_{p'}$.
\end{cor}

\begin{proof}
The lower bound is given by Lemma~\ref{lem:X/e}(c). The upper bound follows
by Lemma~\ref{lem:2e} from the fact that $|T(H):X(H)|=2$.
\end{proof}

\subsection{Location of simple endotrivial modules on the Brauer tree}

Let $\bB$ be a block of $kG$ and let $\sigma({\bB})$ denote its Brauer tree.
Using notation of \cite{HL}, the nodes of $\sigma({\bB})$ can be labelled by
$\textit{noughts}$ $\circ$ and $\textit{crosses}$ $\times$, such that a nought can only be
joined to a cross and a cross to a nought. Moreover, let $\chi_{0}$ denote
the exceptional node of $\sigma({\bB})$, if it has exceptional multiplicity
$m_{\bB}:=(p^{n}-1)/e_{\bB}>1$ and
otherwise let $\chi_{0}$ be any node of $\sigma({\bB})$. If $S$ is a simple
$\bB$-module labelling an edge of $\sigma({\bB})$, let $n(S)$ be the number
of nodes of $\sigma({\bB})$ which are not connected to $\chi_{0}$ after
removal of the edge $S$.  Let $l(f(S))$ denote the length of the $kH$-Green
correspondent of $S$. Then the following holds:

\begin{lem} \label{lem:length1}
 Let $S$ be a simple $\bB$-module labelling an edge of $\sigma({\bB})$ and
 let $\chi$ be the node adjacent to $S$, which is not connected to $\chi_{0}$
 after removal of the edge $S$. Then
 $$l(f(S))  =\begin{cases}
                    n(S)        & \text{if $\chi$ has type $\times$}, \\
                    p^{n}-n(S)  & \text{if $\chi$ has type $\circ$}\,.
 \end{cases}$$
\end{lem}

\begin{proof}
See \cite[Lem.~4.4.11]{HL} and \cite[Lem.~9.3]{Fe}.
\end{proof}

If $S$ is a simple $\bB$-module, let $\times(S)$ denote the number of nodes on
the part of $\sigma(\bB)$ to which the node adjacent to $S$ of type $\times$
belongs after removal of the edge $S$, with the exceptional node counted
$(p^{n}-e)$-times.

\begin{lem} \label{lem:length2}
 Let $S$ be a simple $\bB$-module. Then $l(f(S))=\times(S)$.
\end{lem}

\begin{proof}
\cite[Lem.~4.4.12]{HL} states and proves the case $n=1$. Their proof can be
generalized to an arbitrary $n$: Let $S$ be an edge of $\bB$ and let $\chi$
be the node adjacent to $S$, which is not connected to $\chi_{0}$ after
removal of $S$. If $\chi$ is of type $\times$, then $\times(S)=n(S)=l(f(S))$
by Lemma~\ref{lem:length1}. If $\chi$ is of type $\circ$, then by
Lemma~\ref{lem:length1} we get
$$l(f(S))=p^{n}-n(S)=p^{n}-(e+1- (\times(S)-(p^{n}-e)+1))= \times(S).$$
\end{proof}

\begin{lem}   \label{lem:length3}
 Let $\bB$ be a block of $kG$ containing an endotrivial module and let $S$
 be a simple $\bB$-module. Then $S$ is endotrivial if and only
 if $l(f(S))\in\{1,p^{n}-1\}$.
\end{lem}

\begin{proof}
Let $\bb$ be the Brauer correspondent of $\bB$. Then the claim follows from
the fact that the boundary $\Omega^{2}$-orbits of $\Gamma_{s}(\bb)$ are
made of the indecomposable $kH$-modules of length $1$ and $p^{n}-1$, together
with Lemma~\ref{lem:2e} and~\ref{lem:X/e}.
\end{proof}

Let us call a leaf of a Brauer tree $\sigma(\bB)$ an \emph{exceptional leaf}
if the exceptional node is sitting at the end of this leaf and has exceptional
multiplicity $m_{\bB}> 1$. Then we can state the main result of this section.

\begin{thm}   \label{thm:ETleaves}
 Let $\bB$ be a block of $kG$ containing an endotrivial module and assume
 $e>1$. Let $S$ be a simple $\bB$-module. Then $S$ is endotrivial if and only
 if $S$ corresponds to a non-exceptional leaf of $\sigma(\bB)$.
\end{thm}

\begin{proof}
By definition the leafs of $\sigma(\bB)$ correspond to liftable simple
$\bB$-modules whereas the inner edges correspond to non liftable simple
modules. Hence by Theorem~\ref{thm:lift} only leaves of $\sigma(\bB)$ can
be endotrivial. So let $S$ be a simple $\bB$-module corresponding to a leaf
of $\sigma(\bB)$ and let $\chi$ denote the end node of this leaf. Then by
Lemma~\ref{lem:length3}, $S$ is endotrivial if and only if
$l(f(S))\in\{1,p^{n}-1\}$. Applying Lemma~\ref{lem:length2}, we obtain that
the length $l(f(S))$ is as follows: \par
If $\chi\neq \chi_{0}$ and is of type $\times$, then $l(f(S))=1$. \par
If $\chi\neq \chi_{0}$ and is of type $\circ$, then
$l(f(S))=(e-1)+p^{n}-e=p^{n}-1$.\par
Notice that if $\chi_{0}$ has multiplicity one, then we may always assume that
$\chi\neq \chi_{0}$. Therefore, now assume that $\chi_{0}$ has multiplicity
$m_{\bB}>1$, so $1<e < p-1$.\par
If $\chi= \chi_{0}$ and $\chi$ is of type $\times$, then
$l(f(S))=p^{n}-e\notin\{1,p^{n}-1\}$.\par
If $\chi= \chi_{0}$ and $\chi$ is of type $\circ$, then
$l(f(S))=e\notin\{1,p^{n}-1\}$. Hence the result.
\end{proof}

\subsection{$\SL_{2}(q)$ in cross characteristic}

As an application we classify simple endotrivial modules for $G=\SL_{2}(q)$,
$q=p^{n}$, $p$ a prime, in non-defining characteristic $\ell$ dividing the
order of $G$. For the exceptional covering groups of $\PSL_2(9)\cong\fA_6$ see
Theorem~\ref{thm:mainAn}. For $\ell=p$, see Theorem~\ref{thm:defchar}.

\begin{prop}   \label{prop:SL2}
 Let $G=\SL_{2}(q)$, $q=p^{n}$ with $p$ a prime. Let $V$ be a non-trivial simple
 $kG$-module, where $k$ is algebraically closed of characteristic $\ell\neq p$.
 Then $V$ is endotrivial if and only if one of:
 \begin{enumerate}
  \item[\rm(1)] $2\neq \ell \, | \, q-1$ and $V$ lies in an
   $\ell$-block of full defect and inertial index 2 (cyclic defect); 
  \item[\rm(2)] $p\neq 2\neq \ell\,|\,q+1$ and $V$ lies in the non-principal
   $\ell$-block of full defect and inertial index 2 (cyclic defect); 
  \item[\rm(3)] $3=\ell \, | \, q+1$, $|G|_\ell=3$ and $V$ lies in
   the principal $\ell$-block (cyclic defect).
 \end{enumerate}
 Moreover, if $\ell=2$, $q\equiv -1\pmod{4}$  and $V$ lies in the principal
 $\ell$-block, then $V$ is endotrivial as a $k\PSL_{2}(q)$-module, but not
 as a $kG$-module.
\end{prop}

This proposition is essentially proven in \cite{Schu} via character theory.
Only the cases $p\neq \ell=2$, $q\equiv-1\pmod{4}$ and
$p\neq 2\neq \ell\,|\,q+1$ for the non-principal $\ell$-block of full
defect and inertial index 2 were left open. We give here a summary proof using
the techniques of this section and some results of Craven's \cite{Cr}.

\begin{proof}
Only blocks with full defect need to be investigated, for indecomposable
endotrivial modules have Sylow subgroups as vertices. In addition, in the
cyclic defect case, only blocks $\bB$ with $e_{\bB}=e$ can contain endotrivial
modules by Lemma~\ref{lem:2e}. Let $\bB_{0}$ denote the principal $\ell$-block
of $kG$. Recall that $|\SL_{2}(q)|=(q-1)q(q+1)$. \par
If $\ell\neq 2$ and $\ell \, | \, q-1$ or $\ell \, | \, q+1$, then a
Sylow $\ell$-subgroup of $G$ is cyclic and $e=2$. Moreover if $p\neq 2$, then
$G$ has exactly two
$\ell$-blocks of full defect and inertial index 2, $\bB_{0}$ and say $\bB_{1}$.
By \cite[\S 6.2.1 and \S 6.2.2]{B11}, $|X(H)|=4$, thus the number of blocks
containing endotrivial modules is $|X(H)|/e=2$ by Lemma~\ref{lem:X/e}.
Now if $\ell\,|\,q-1$, the Brauer trees $\sigma(\bB_{0})$, $\sigma(\bB_{1})$
both have the form $\xymatrix{ \circ \ar@{-}[r]&\bullet\ar@{-}[r] &\circ}$
with exceptional node in the middle. (See \cite[Sec.~9.3 and \S 9.4.2]{B11}).
Therefore, by Theorem~\ref{thm:ETleaves}, all non-trivial simple $\bB_{0}$-
and $\bB_{1}$-modules are endotrivial. If $\ell\,|\,q+1$, $\sigma(\bB_{1})$
has the form $\xymatrix{ \circ \ar@{-}[r]&\bullet\ar@{-}[r] &\circ}$ with
exceptional node in the middle (see \cite[Sec.~9.3]{B11}), hence both simple
$\bB_{1}$-modules are endotrivial by Theorem~\ref{thm:ETleaves}.
The tree $\sigma(\bB_{0})$ has the form
$\xymatrix{ \circ \ar@{-}[r]&\circ\ar@{-}[r] &\bullet}$ with exceptional node
sitting on one end (see \cite[\S 9.4.3]{B11}). Moreover the exceptional
multiplicity is $m_{\bB_{0}}=1$ if and only if $\ell=3$ and $|G|_\ell=3$.
Thus, by Theorem~\ref{thm:ETleaves}, $\bB_{0}$ contains a non-trivial simple
endotrivial module if and only if $|G|_\ell=3$. (The trivial module corresponds
to the non-exceptional end node.) Now if $p=2$, the situation is similar,
except that only the principal block has full defect and inertial index~2
(see \cite[\S 9.4.2]{B11}). \par
If $p\neq \ell=2$, then only the principal block has full defect: it contains
two non-trivial simple modules $S_{1}$ and $S_{2}$ with
$\dim(S_{1})=\dim(S_{2})=\frac{1}{2}(q-1)$ (see \cite[Chap. 9]{B11}).
If $q\equiv 1\pmod{4}$, then $\dim(S_i)\equiv 0\pmod{2}$. Hence $S_{1}$,
$S_{2}$ are not endotrivial by Lemma~\ref{lem:deg}. If $q\equiv -1\pmod{4}$,
then $S_{1}$ and $S_{2}$ are endotrivial $k\PSL_{2}(q)$-modules by
\cite[Sec.~4.4]{Cr} for $q\equiv 3\pmod{8}$ and by \cite[Prop.~4.26]{Cr} for
$q\equiv 7\pmod{8}$. However, the inflation of $S_{1}$ and $S_{2}$ from
$\PSL_{2}(q)=G/Z(G)$ to $G$ does not yield endotrivial $kG$-modules. Indeed,
since $|Z(G)|=2$, we get for $1\leq i\leq 2$, $\End_{k}(S_{i})\cong
k\oplus( \Ind_{Z(G)}^{G}(k)\oplus\cdots\oplus \Ind_{Z(G)}^{G}(k))$ as
$kG$-modules.
\end{proof}

\section{Covering groups of alternating groups} \label{sec:alt}

In this section we classify simple endotrivial modules for covering groups
of alternating groups. Note that some (but not all) simple endotrivial modules
for alternating groups were described in \cite{CHM10,CMN09}, using different
methods. The faithful modules for proper covering groups were not investigated
previously.

Throughout $p$ denotes a prime and $k$ a large enough
field of characteristic~$p$.

\subsection{Faithful modules for $\fA_n$}
Recall that the simple $\QQ\fS_n$-modules are parametrized by partitions of $n$.
We first classify partitions possessing certain types of hooks.

\begin{lem}   \label{lem:Sn}
 Let $n=mp+r$ with $0\le r<p$ and assume that $n\ge 2p$. Let
 $\la\vdash n$ be a partition such that the corresponding irreducible
 character $\chi_\la$ of $\fS_n$ does not vanish on elements of cycle shape
 $(n-p)(p)$, $(n-p-1)(p)(1)$, nor on elements whose cycle shape contains a
 cycle of length $n-r$. Then $\la$ or its conjugate is one of the partitions
 $$(1^n)\quad\text{ or }\quad  (n-p,r+1,1^{p-r-1}).$$
\end{lem}

\begin{proof}
By the Murnaghan--Nakayama formula (see e.g.~\cite[2.4.7]{JK81}), if $\chi_\la$
does not vanish on
elements whose cycle shape contains an $n-r$-cycle, then $\la$ must have a
hook of length $n-r$, and the partition $\mu$ obtained by removing such a
hook must have hooks of any length less or equal to $r$. Thus, by assumption
$\mu$ is, up to conjugate, the partition $(r)$, and then $\la$ is of the form
$$(r,b,1^{n-r-b})\quad(1\le b\le r)\quad\text{or}\quad
  (r+a,r+1,1^{n-2r-a-1})\quad(1\le a\le n-2r-1),$$
or $\la$ is the hook $(r,1^{n-r})$. \par
If $\chi_\la$ also vanishes on elements of cycle shape $(n-p)(p)$, then
$\la$ has to possess a hook of length $n-p$, and the partition obtained by
removing such a hook must be a hook itself (of length $p$). It is easily
seen that for the above possibilities, if $\la=(r,b,1^{n-r-b})$ then we have
$r=p-1$, so $\la=(p-1,1^{n-p+1})$, if $\la=(r+a,r+1,1^{n-2r-a-1})$ then only
the two partitions $(n-p,r+1,1^{p-r-1})$ and $(p,r+1,1^{n-p-r-1})$ are
possible, and the only hook satisfying our condition is $(r,1^{n-r})$.\par
If finally $\chi_\la$ does not vanish on elements of cycle shape $(n-p-1)(p)$,
then $\la$ has an $n-p-1$-hook, and the remaining partition has a $p$-hook.
In our first case, this is seen not to be possible. In the second
case, it holds true when $\la=(n-p,r+1,1^{p-r-1})$, and finally, the hook
$(r,1^{n-r})$ only has this property when $r=1$, so $\la=(1^n)$. The claim
is shown.
\end{proof}

\begin{prop}   \label{prop:Anodd}
 Any simple faithful endotrivial $k\fA_n$-module, $n\ge 2p$, over a
 field $k$ of characteristic $p>2$ is a constituent of the restriction to
 $\fA_n$ of a $\QQ\fS_n$-module indexed by $\la\vdash n$, where one of:
 \begin{enumerate}
  \item[\rm(1)] $2p\le n=2p+r\le 3p-1$ and $\la=(p+r,r+1,1^{p-r-1})$;
  \item[\rm(2)] $n=2p$, $\la=(p,2,1^{p-2})$; or
  \item[\rm(3)] $n=2p+1$, $\la=(p+1,1^p)$.
 \end{enumerate}
\end{prop}

\begin{proof}
By Clifford-theory any simple faithful $\cO\fA_n$-module occurs in the
restriction of some simple faithful $\cO\fS_n$-module $V$. We distinguish two
cases. First assume that $V$ restricts irreducibly. Then according to
Lemma~\ref{lem:sub} it suffices to show that $V$ is not endotrivial for $\fS_n$
to conclude the same for $\fA_n$. For this we will show that any non-linear
character $\chi_\la$ of $\fS_n$ except for those listed in~(1) of the claim
vanishes on some $p$-singular element, from which the assertion will follow
by Corollary~\ref{cor:zero}. \par
Note that permutations of cycle shape $(n-p)(p)$, $(n-p-1)(p)$ and $(n-r)(r)$
are $p$-singular, where $n=mp+r$, $0\le r<p$. Thus, our claim already holds by
Lemma~\ref{lem:Sn} unless $\la=(n-p,r+1,1^{p-r-1})$ or the conjugate
partition. But $\la$ does not have a hook of length $n-r-p$ if $n\ge 3p$, so
$\chi_\la$ vanishes on $p$-singular elements of cycle shape $(n-r-p)(p)(1)^r$.
This only leaves the values of $n$ listed in case~(1).
\par
Now we deal with the case that $V$ does not restrict irreducibly to $\cO\fA_n$.
Then $V$ is a simple $\cO\fS_n$-module indexed by a self-conjugate partition
(see e.g.~\cite{}). Here we will show that unless we are in cases~(2) or~(3)
the corresponding character $\chi_\la$ vanishes on some
$p$-singular conjugacy class $C$ of $\fS_n$ contained in $\fA_n$ which forms a
single $\fA_n$-class. Then both constituents of $\chi_\la|_{\fA_n}$ will
vanish on $C$, and again we are done by Corollary~\ref{cor:zero}. \par
Recall that a class of $\fS_n$ splits into two $\fA_n$-classes if and only if
its elements have a cycle shape consisting of odd cycles of mutually distinct
lengths. First assume that $n$ is odd and $r$ is even. If $\chi_\la$ belongs to
an endotrivial module, then by Corollary~\ref{cor:zero} 
it cannot vanish on elements of cycle shapes $(n-p-2)(p)(1)^2$,
and $(n-r)(r-2)(1)^2$ when $r>1$, respectively $(n-p-3)(p)(1)^3$ when $r=1$,
which forces $\la$ to possess $n-p-2$-hooks and moreover either $n-r$ and
$r-2$-hooks, or $n-p-3$-hooks. But there are no such self-conjugate
partitions. Similarly, when $n$ is even and $r$ is odd, we look at the values
on elements of cycle shapes $(n-r)(r-2)(2)$ (resp. $(n-2)(1)^2$ when $r=2$),
$(n-p-2)(p)(2)$, and $(n-p)(2)$, $(2p-2)(p)(2)$ when $r=0$, to see that there
are no relevant self-conjugate $\la$. \par
If $n$ and $r$ are both even, the values on elements of type $(n-r)(r)$ for
$r>0$, respectively of type $(n-p)$, $(n-p-2)$ and $(n-p-3)(2)$ show that
only $\la=(p,2,1^{p-2})$ can possibly index an endotrivial module, in which
case moreover $n=2p$. Finally,
for $n$ and $r$ both odd, we argue with the cycle shapes $(n-r)(r-1)$ (for
$r>1$) respectively $((n-1)/2)^2$, $(n-p-1)(1)^{p+1}$, $(n-p-3)(p)(1)^3$ to
see that necessarily $n=2p+1$, and $\la$ must either be as in (3), or
$p=5$, $\la=(4,3^2,1)$, or $p=7$ and $\la=(4^3,3)$. When $\la=(4,3^2,1)$, then
the corresponding character vanishes on elements of cycle shape $(5)(3)^2$,
while for $\la=(4^3,3)$, it vanishes on elements of cycle shape $(7)(6)(2)$.
This completes the proof.
\end{proof}

\begin{prop}   \label{prop:An2}
 The group $G=\fA_n$ does not have simple faithful endotrivial
 $kG$-modules for $k$ a field of characteristic $2$ when $n\ge 8$.
\end{prop}

\begin{proof}
We show that any irreducible character $\chi_\la$ of $\fS_n$ vanishes on some
even order element contained in $\fA_n$ (which is sufficient since
$\fS_n$-classes of even order elements never split considered as
$\fA_n$-classes). \par
First assume that $n$ is even. For $n=8$ we have $\fA_8=\PSL_4(2)$, so the
claim follows from Theorem~\ref{thm:defchar}; for $n=10$ it can be checked
from the known character table of $\fS_{10}$.
For $n\ge12$ we use that $\fA_n$ contains
elements of cycle shapes $(n-2)(2)$, $(n-4)(4)$ and $(n-5)(2^2)$ to see that
$\la$ necessarily has to possess hooks of lengths $n-2$, $n-4$ and $n-5$. Up
to conjugates this only leaves $\la=(n-1,1)$, $(n-2,2)$ and $(n-4,3,1)$.
But the characters indexed by these partitions vanish on elements of
cycle shapes $(n-5)(2^2)$, $(n-6)(6)$, $(n-6)(6)$ respectively. \par
If $n$ is odd, then the cases $n=9,11$ can again be checked from the character
tables. Now let $n\ge13$. Then non-vanishing on elements of cycle shapes
$(n-3)(2)$, $(n-4)(2^2)$, and $(n-5)(4)$ implies that $\la$ has hooks of
lengths $n-3$, $n-4$ and $n-5$, whence $\la=(n-1,1)$, $(n-2,1^2)$ or
$(n-4,4)$ up to conjugates. But the corresponding characters vanish on elements
of cycle shapes $(n-3)(2)$, $(n-5)(4)$, $(n-7)(6)$ respectively.
\end{proof}

Let us next consider the cyclic defect cases:

\begin{prop}   \label{prop:Ancyc}
 Let $G=\fA_n$ with $5\le n\le p<2p$.
 \begin{enumerate}
  \item[\rm(a)] If $n=p,p+1$ then $kG$ has no non-trivial simple endotrivial
   module.
  \item[\rm(b)] If $p+2\le n<2p$ then $kG$ has exactly one non-trivial simple
   endotrivial module, namely the Specht module indexed by the partition
   $(p+1,1^{n-p-1})$.
 \end{enumerate}
\end{prop}

\begin{proof}
Let $P$ be a Sylow $p$-subgroup of $G$ and $H=N_{\fA_n}(P)$. By our assumptions
on $n$, $P$ is cyclic of order~$p$. Write $n=p+r$, with $0\le r<p$. Then
$N_{\fS_n}(P)\cong (C_p\rtimes C_{p-1})\times\fS_r$ and it follows that
$$e=|H:C_{\fA_n}(P)|= \begin{cases}
     \frac{p-1}{2}&  \text{ if } n=p,p+1,\\
               p-1&  \text{ if } p+2\le n<2p.
\end{cases}$$
Furthermore, an easy computation shows that
$$|X(H)|= \begin{cases}
     \frac{p-1}{2}&  \text{ if } n=p,p+1, \\
               p-1&  \text{ if } p+2\le n<2p.
\end{cases}$$
In both cases $|X(H)|=e$, meaning that all the indecomposable endotrivial
modules lie in the principal block $\bB_{0}$ of $\fA_n$. The Brauer tree of
$\bB_{0}$ is a straight line, with exceptional node sitting on one
end in case $n=p,p+1$.
Thus the claims in~(a) and in the
first part of~(b) are a direct consequence of Theorem~\ref{thm:ETleaves}.
It follows from the explicit knowledge of the Brauer tree that for $n\ge p+2$
the end node corresponds to the Specht module indexed by the hook partition
$(p+1,1^{r-1})$. This gives the remaining assertion in~(b).
\end{proof}

\subsection{Faithful modules for $\tfA_n$}
We now discuss faithful simple $kG$-modules for $G=\tilde\fA_n$, $n\ge5$, the
double covering group of $\fA_n$, with center of order~2. \par
We recall some facts from the ordinary representation theory of $\tfS_n$,
where $\tfS_n$ denotes any of the two double covering groups of $\fS_n$.
Let $\cD(n)$ denote the set of partitions of $n$ into distinct parts. We say
that a partition $\la$ is \emph{odd} if its number of even parts is odd, and
else we call it \emph{even}. \par
The faithful complex irreducible characters of $\tfS_n$ are parametrized by
partitions $\la\in\cD(n)$ as follows: if $\la$ is even there is one
irreducible character $\psi_\la\in\Irr(\tfS_n)$ which splits upon restriction
to $\tfA_n$ into two distinct constituents $\psi_\la^\pm$; if $\la$ is odd,
there are two irreducible characters $\psi_\la^\pm\in\Irr(\tfS_n)$ which have
the same restriction to $\tfA_n$, see e.g. \cite[Thm.~8.6]{HH}.

\begin{thm}   \label{thm:tildeAn}
 The group $G=\tfA_n$ does not have faithful simple endotrivial
 $kG$-modules for $k$ a field of characteristic $p>0$ when $n\ge\min\{2p,p+4\}$.
\end{thm}

\begin{proof}
Note that $Z(G)$ lies in the kernel of any simple 2-modular $kG$-module, so
we may certainly assume that $p>2$. Let first $\la\in\cD(n)$ be odd and assume
that $n\ge p+4$. Then the irreducible characters $\psi_\la^\pm$ of $\tfS_n$
restrict irreducibly to $\tfA_n$, so we may argue in $\tfS_n$ by
Lemma~\ref{lem:sub}. Let $g\in\tfS_n$ be an element whose projection to
$\fS_n$ has cycle shape $\mu=(p)(2)(1)^{n-p-2}$. Then $\psi_\la^\pm$ vanishes
on $g$ unless $\la=\mu$, by the theorem of Schur (see \cite[Thm.~8.7(ii)]{HH}),
while when $\la=\mu$, $\psi_\la^\pm$ vanishes on elements of cycle shape
$(p)(4)(1)^{n-p-4}$. Since these elements are $p$-singular, $\psi_\la^\pm$
cannot be endotrivial by Corollary~\ref{cor:zero}. When $n=2p< p+4$, so
$p=3$, $n=6$, the only candidate is $\psi_\la^\pm$ with $\la=(3,2,1)$, and this
vanishes on (3-singular) elements of cycle shape $(6)$.
\par
If $\la\in\cD(n)$ is even, we need to look at the constituents $\psi_\la^\pm$
of the restriction of $\psi_\la$ to $\tfA_n$. By \cite[Thm.~8.7]{HH},
$\psi_\la^+(g)=\psi_\la^-(g)$ whenever $g$ has cycle shape different from
$\la$. Now elements with cycle shape $(p)(2)^2(1)^{n-p-4}$ are contained in
$\tfA_n$ and $\psi_\la$ vanishes on these again by \cite[Thm.~8.7(iii)]{HH}.
\end{proof}

Note that faithful simple modules for the exceptional six-fold covering groups
$6.\fA_6$, $6.\fA_7$ only exist in characteristic $p\ge5$, so for covering
groups of $\fA_n$ with centre of even order we are only left with cases with
cyclic Sylow $p$-subgroup. We will deal with the exceptional covering groups
in the next subsection.

\begin{prop}   \label{prop:2Ancyc}
 Let $G=\tfA_n$ with $5\le p\le n\le p+3$. Then the faithful simple
 endotrivial $kG$-modules are precisely those indexed by the following
 partitions:
 \begin{enumerate}
  \item[\rm(1)] $((p+1)/2,(p-1)/2)$ when $n=p$;
  \item[\rm(2)] $(p+1)$  and $((p+1)/2,(p-1)/2,1)$ when $n=p+1$;
  \item[\rm(3)] $(p+2)$ (two non-isomorphic modules) and, for $p>5$,
   $((p+1)/2,(p-1)/2,2)$ (two non-isomorphic modules) when $n=p+2$; and
  \item[\rm(4)] $(p+2,1)$ (two non-isomorphic modules) and, for $p>5$,
   $((p+1)/2,(p-1)/2,2,1)$ (two non-isomorphic modules) when $n=p+3$.
 \end{enumerate}
 In characteristic $p=3$, $k\tfA_5$ has two faithful simple endotrivial
 modules, both of dimension~2.
\end{prop}

\begin{proof}
It is easily seen that for $5\le p\le n\le p+3$ the normalizer $H$ of a (cyclic)
Sylow $p$-subgroup of $\fA_n$ has cyclic Sylow 2-subgroups, so the normalizer
$\tilde H$ of a Sylow $p$-subgroup of $\tfA_n$ has abelian Sylow 2-subgroups.
In particular, $|X(\tilde H)|=2|X(H)|$, and thus by
Corollary~\ref{cor:bounds}(c) there exists exactly one non-trivial faithful
$p$-block of $\tfA_n$ containing simple endotrivial modules. (In contrast, for
$n\ge p+4$ we have $|X(\tilde H)|=|X(H)|$ and so there do not exist faithful
simple endotrivial modules, in accordance with Theorem~\ref{thm:tildeAn}.)
\par
First assume that $p\ge5$. The Brauer trees for the faithful blocks of
$\tfA_n$ have been calculated by M\"uller \cite{Mue03}. In our situation,
there are five blocks to consider,
corresponding to the $p$-bar cores $()$, $(1)$, $(2)$, $(2,1)$, and~$(3)$.
These have associated sign and $s$-invariant $(+,0)$, $(+,1)$, $(-,2)$,
$(-,2)$, $(+,1)$ respectively. By \cite[Thm.~4.4]{Mue03} the Brauer tree is a
straight line in the cases of sign ``$+$'', or with sign ``$-$'' and $p=5$,
with the exceptional node at the end when $s=0$. Moreover, the end nodes are
as given in the statement. Otherwise, the Brauer tree is a star with four arms
of positive length and no exceptional node. \par
For the $p$-bar core $(3)$, the degree of the character parametrized by the
partition $(p+3)$ is twice the degree of the character for the $p$-bar core
$(2)$ parametrized by the partition $(p+2)$ (see \cite[Thm.~10.7]{HH}) and
hence not congruent to $\pm1\pmod p$, whence this block cannot contain simple
endotrivial modules. Thus, for $n=p+3$ the block with associated $p$-bar core
$(2,1)$ is the one with simple endotrivial modules. Again the labels of the
end nodes can be read off from \cite[Thm.~4.4]{Mue03}.
\end{proof}

It ensues from \cite[Thm.~10.7]{HH} that the dimensions of the endotrivial
modules in Proposition~\ref{prop:2Ancyc} are given as follows, where
$m:=(p-1)/2$:

\[\begin{array}{|c|lll|}
\hline
 n& \chi_\la(1)& & \\
\hline
 p&     2^{m-1}\binom{p-1}{m}& & \\
 p+1&   2^{m}& \text{ and }& 2^{m}\binom{p-1}{m-2}\\
 p+2&   2^{m}& \text{ and }& 2^{m}\binom{p-1}{m-3}\frac{(p+1)(p+2)}{p-1}\\
 p+3&   2^{m}(p+1)& \text{ and }& 2^{m-1}\binom{p-1}{m-3}\frac{(p-3)(p+2)}{3}\\
\hline
\end{array}\]

\subsection{The result for alternating groups}

\begin{lem}   \label{lem:An1}
 Let $n=2p+1>5$, and $V$ be a $k\fA_n$-constituent of the (irreducible) Specht
 module for $k\fS_n$ indexed by the partition $\la=(p+1,1^p)$. Then $V$ is
 endotrivial.
\end{lem}

\begin{proof}
By \cite[Thm.~2]{Pe69} the Specht module for $\la$ is irreducible modulo~$p$.
Let $H=\fA_{p+1}\times\fA_p$ denote a Young subgroup of $G=\fA_n$. Then $H$
contains a Sylow $p$-subgroup of $G$. Thus by Lemma~\ref{lem:sub} it suffices
to show that the restriction $V|_H$ of $V$ to $H$ is endotrivial. The
Littlewood--Richardson rule shows that the ordinary character of $V|_H$ is
given by
$$1 \boxtimes 1+\sum_{i=1}^{(p-1)/2} \chi_i \boxtimes(\psi_i+\psi_{i+1}),$$
where $\chi_i$ denotes the character indexed by the hook partition
$(p+1-i,1^i)$ and $\psi_i$ the character indexed by the hook partition
$(p+1-i,1^{i-1})$. (Here, $\boxtimes$ is the external tensor product.) The only
constituent in the principal block is the trivial
character, and all the $\chi_i$ are of defect zero for $\fA_{p+1}$, so sorting
by blocks we get a direct decomposition
$$V|_H=k\oplus \bigoplus_{i=1}^{(p-1)/2} P_i \boxtimes M_i,$$
where $P_i$ is the projective $k\fA_{p+1}$-module with character $\chi_i$ and
$M_i$ is a $k\fA_p$-module with character $\psi_i+\psi_{i+1}$.
It hence suffices to argue that all $M_i$ are projective $k\fA_p$-modules. We
have already seen that $V|_{\fA_{p+1}}=1+$(projectives), so its restriction to
$\fA_p<\fA_{p+1}$ also has this form. But the Young subgroup $\fA_p$ of
$\fA_{p+1}$ is conjugate to the second factor of $H$, so we conclude that
$\sum_i M_i$ is projective, and hence that each $M_i$ is.
\end{proof}

This has also been shown in \cite[Prop.~8.3]{CMN09} by a more involved
argument.

\begin{lem}   \label{lem:An2}
 Let $n=3p-1>5$, and $V$ be the $p$-modular reduction of the Specht module of
 $\fS_n$ indexed by the partition $\la=(2p-1,p)$. Then $V$ is simple and
 endotrivial.
\end{lem}

\begin{proof}
The Specht module for the partition $(2p-1,p)$ is irreducible modulo~$p$ by
a theorem of Fayers and isomorphic to the Young module $Y^{(2p-1,p)}$
(see \cite[Prop.~1.1]{He06}). Restriction of $Y^{(2p-1,p)}$ to the Young
subgroup $\fS_{3p-3}$ gives
$$Y^{(2p-1,p)}|_{\fS_{3p-3}}=Y^{(2p-3,p)}\oplus 2 Y^{(2p-2,p-1)}$$
by two-fold application of \cite[Thm.~5.1]{He05}. As $(2p-2,p-1)$ is
$p$-restricted, the Young module $Y^{(2p-2,p-1)}$ is projective by
\cite[Thm.~2]{Er01}, while $Y^{(2p-3,p)}$ is endotrivial by
\cite[Prop.~8.2]{CMN09}. The claim follows since $\fS_{3p-3}$ contains a
Sylow $p$-subgroup of $\fS_{3p-1}$.
\end{proof}

The proof of the endotriviality of $Y^{(2p-1,p)}$ given in \cite{CMN09} seems
unclear to us. 

\begin{thm}   \label{thm:mainAn}
 Let $V$ be a faithful simple $kG$-module, for some covering group
 $G$ of $\fA_n$, $n\ge \max\{p,5\}$, over a field $k$ of characteristic $p>0$.
 Then $V$ is endotrivial if and only if $V$ is a constituent of the simple
 module for the corresponding covering group of $\fS_n$ indexed by
 $\la\vdash n$, where one of:
 \begin{enumerate}
  \item[\rm(1)] $G=\fA_n$, $5\le p+2\le n<2p$ and $\la=(p+1,1^{n-p-1})$
   (cyclic defect);
  \item[\rm(2)] $G=\tfA_n$, $3\le p\le n\le p+3$ and $\la$ is as in
   Proposition~\ref{prop:2Ancyc} (cyclic defect);
  \item[\rm(3)] $G=\fA_n$, $p>2$, $n=2p+1$ and $\la=(p+1,1^p)$;
  \item[\rm(4)] $G=\fA_n$, $p>2$, $n=3p-1$ and $\la=(2p-1,p)$; or
  \item[\rm(5)] $n=6,7$, $|Z(G)|\ge3$ and $(G,p,V)$ are as in
   Table~\ref{tab:6A}.
 \end{enumerate}
\end{thm}

\begin{proof}
By Theorem~\ref{thm:tildeAn} and Proposition~\ref{prop:2Ancyc} the only examples
with $|Z(G)|=2$ are those in~(2) of the conclusion.  Let us now assume that
$G=\fA_n$. If $n<2p$, the Sylow $p$-subgroups of $G$ are cyclic. Then
Proposition~\ref{prop:Ancyc} gives~(1) of the conclusion. \par
So now assume that $G=\fA_n$, $n\ge 2p$ and $p>2$. Then $V$ is an
$\fA_n$-constituent of the
$p$-modular reduction of the Specht module indexed by one of the partitions
$\la\vdash n$ listed in Proposition~\ref{prop:Anodd}. First consider the
possibilities in case~(1) of that result. Note that the $p$-core of these
$\la$ is the partition $(r)$, so $V$ lies in the principal $p$-block.
By James \cite[Cor.~2.11]{J78b} the only Specht module in the principal block
which remains irreducible modulo~$p$ is the trivial module, unless
$n\equiv-1\pmod p$. Clearly the same holds for the module indexed by the
conjugate partition, which is obtained by tensoring with the sign
representation. Thus, we only need to consider the case $r=p-1$,
$\la=(2p-1,p)$. The corresponding Specht module is endotrivial by
Lemma~\ref{lem:An2}, leading to case~(4).
By the same criterion, the Specht module for $\la=(p,2,1^{p-2})$ as in
Proposition~\ref{prop:Anodd}(2) is reducible. Finally, the Specht module
parametrized by the partition $\la=(p+1,1^p)$ in
Proposition~\ref{prop:Anodd}(3) is in fact an example by Lemma~\ref{lem:An1}.
\par
Now assume that $G=\fA_n$ and $p=2$. Then there are no examples for $n\ge8$ by
Proposition~\ref{prop:An2}. For $\fA_5=\PSL_2(4)$, there's no example by
Theorem~\ref{thm:defchar}, and for $n=6,7$, use of Corollary~\ref{cor:zero}
and the character tables shows that no cases arise. \par
Finally, if $G$ is one of the 3- or 6-fold exceptional covering groups of
$\fA_6$ or $\fA_7$, then the claim for $p=5,7$ follows by the cyclic defect
methods from Section~\ref{sec:cycl}, while for $p=2$ the character tables
in \cite{Atl} show that the only faithful candidate characters are those
listed in Table~\ref{tab:6A}. Explicit calculation with these modules yields
that all candidates are in fact endotrivial.
\end{proof}

\begin{table}[htbp]
 \caption{Relevant blocks in $6.\fA_6$ and $6.\fA_7$ }
  \label{tab:6A}
\[\begin{array}{|r|r|r|c|c|l|}
\hline
 G& p& X(H)& X(H)/e& \text{block}& \dim V\\
\hline\hline
 3.\fA_6& 2&  -&  -&    4,5&   3,3,9\\
\hline\hline
 3.\fA_6& 5&  6&  3&    5,6&      6\\
 6.\fA_6& 5& 12&  6&  12,13&      6, 6\\
\hline\hline
 3.\fA_7& 5& 12&  3&    6,7&      6, 21\\
 6.\fA_7& 5& 24&  6&  15,16&      6, 6, 24\\ \cline{2-6}
 3.\fA_7& 7&  9&  3&    6,7&      6, 15\\
 6.\fA_7& 7& 18&  6&  15,16&      6, 6\\
\hline
\end{array}\]
\end{table}

\section{Groups of Lie type in defining characteristic} \label{sec:defchar}

In this section we classify the simple endotrivial modules for quasi-simple
groups of Lie type in their defining characteristic. It turns out that these
are extremely rare: they only occur for rank at most~2. Carlson, Mazza and
Nakano \cite{CMN06} determined the structure of the group of endotrivial
modules in these cases. Again, our results are independent and do not seem
to follow in an obvious way from this.

Let $G$ be a connected reductive linear algebraic group over the algebraic
closure of a finite field of characteristic~$p$. Let $\Phi$ denote the root
system of $G$ with respect to some maximal torus, and $\Phi^+\subset\Phi$ a
positive system. Let $\la_1,\ldots,\la_l$ denote the corresponding fundamental
dominant weights of $G$, and $\rho=\sum_{i=1}^l\la_i$. We write $N:=|\Phi^+|$.
For a dominant weight $\la$ of $G$ we denote by $L(\la)$ the corresponding
simple highest weight module (see e.g. \cite[\S15]{MT}).

\begin{prop}   \label{prop:defchar}
 Let $G$ be as above. Let $\la$ be a $p$-restricted dominant weight of $G$,
 different from the Steinberg weight $(p-1)\rho$. Then:
 \begin{enumerate}
  \item[\rm(a)] $\dim L(\la)<p^N$.
  \item[\rm(b)] $\dim L(\la)<p^N-1$, unless $N=1$ and $\la=(p-2)\rho$.
  \item[\rm(c)] If $p=2$ then $\dim L(\la)<2^{N-1}-1$, unless $N\le3$.
 \end{enumerate}
\end{prop}

\begin{proof}
By Weyl's character formula, the dimension of the corresponding highest weight
module $L_\CC(\la)$ for an algebraic group $G_\CC$ of the same type over the
complex numbers is given by
$$\dim L_\CC(\la)=
   \prod_{\al\in\Phi^+}\frac{\lang\la+\rho,\al\rang}{\lang\rho,\al\rang}
$$
(see \cite[24.3]{Hum}). It is well-known that the dimension of $L_\CC(\la)$
is an upper bound for the dimension of $L(\la)$. For $\la=(p-1)\rho$ the
formula gives
$\dim L_\CC(\la)=\prod_{\al\in\Phi^+}p=p^N$. Since the Steinberg module
$L_\CC(\la)$ remains irreducible under restriction to characteristic~$p$,
this shows that $\dim L(\la)=p^N$. Any other $p$-restricted weight is of
the form $\la=(p-1)\rho-\psi$, where $\psi=\sum_{i=1}^l a_i\la_i\ne0$ is a
non-negative integral linear combination of fundamental weights, with $a_j>0$,
say. Then
$$\begin{aligned}
 \dim L(\la)\le \dim L_\CC(\la)
  =&\prod_{\al\in\Phi^+}\frac{\lang p\rho-\psi,\al\rang}{\lang\rho,\al\rang}
  =\prod_{\al\in\Phi^+} \big(p-\frac{\lang\psi,\al\rang}{\lang\rho,\al\rang}\big)\\
  \le& \prod_{\al\ne\al_j} p\big(p-\frac{\lang\psi,\al_j\rang}{\lang\rho,\al_j\rang}\big)=p^{N-1}(p-a_j),
\end{aligned}$$
using that $\langle\,,\,\rangle$ is linear in the first argument.
Clearly, this is smaller than $p^N-1$ unless $N=1$, $\la=(p-2)\la_1$,
which gives (a) and (b).  \par
If $p=2$, then the above argument shows that $\dim L(\la)\le 2^{N-1}$ when
$\la\ne\rho$. If $N>1$ then since $\Phi$ is indecomposable there is at least
one further positive root $\al_j+\al_m\in\Phi^+$ which involves $\al_j$. Then
we get
$$\dim L(\la)\le 2^{N-2}(2-1)(2-\frac{1}{2})=3\cdot 2^{N-3},$$
which is smaller than $2^{N-1}-1$ when $N>3$.
\end{proof}

We thank Frank L\"ubeck for showing us the proof of~(a).

\begin{thm}   \label{thm:defchar}
 Let $G$ be a finite quasi-simple group of Lie type in characteristic~$p>0$.
 Let $V$ be a simple faithful $kG$-module, where $k$ is algebraically closed
 of characteristic~$p$. Then $V$ is endotrivial if and only if one of
 \begin{enumerate}
  \item[\rm(1)] $p\ge5$, $G=\SL_2(p)$ and $\dim V=p-1$; or
  \item[\rm(2)] $p=2$, $G=\SL_3(2)$ and $\dim V=3$.
 \end{enumerate}
\end{thm}

\begin{proof}
Note that the exceptional Schur multipliers of groups of Lie type all have
order a power of the defining characteristic $p$, so lie in the kernel of
all $kG$-representations. Thus, we may assume that $V$ is a (not necessarily
faithful) non-trivial simple $kH$-module, where $H$ is a group of simply
connected type such that $G=H/Z$ for some central subgroup $Z\le H$.  \par
According to Steinberg's tensor
product theorem (see e.g. \cite[Thm.~16.12]{MT}) the simple $kH$-modules are
tensor products of Frobenius twists of $p$-restricted highest weight modules.
If any of the factors in such a tensor product is a twist of the Steinberg
module, then $\dim V$ is divisible by $p$, hence not endotrivial by
Lemma~\ref{lem:deg}.
All other $p$-restricted highest weight modules have dimension at most $p^N-1$,
and even strictly smaller than this unless $H$ is of type $A_1$, by
Proposition~\ref{prop:defchar}. Thus $V$ has dimension smaller than $q^N-1$,
where $|H|_p=q^N$, unless $H$ is of type $A_1$ and $p=q$. In the former case
$V$ cannot be endotrivial by Lemma~\ref{lem:deg} for $p>2$. \par
If $p=2$ then either $N=3$, in which case $H=\SL_3(2)$ (note that $\SU_3(2)$
is solvable), or $\dim V<2^{N-1}-1$ by Proposition~\ref{prop:defchar}. As
$|H|_2=2^N$, the latter modules cannot be endotrivial by Lemma~\ref{lem:deg}.
The natural representation of $\SL_3(2)\cong\PSL_2(7)$ is an example by
Proposition~\ref{prop:SL2}.  \par
So finally assume that $H=\SL_2(p)$ and $\dim V=p-1$. Let $P\le H$ denote a
Sylow $p$-subgroup. Since $V|_P$ is indecomposable, and up to isomorphism
there is a unique indecomposable $kP$-module of dimension $p-1$, namely
$\Omega(k)$, we must have $V|_P\,\cong\,\Omega(k)$, which is endotrivial.
Thus this is indeed an endotrivial module by Lemma~\ref{lem:sub}.
\end{proof}

\section{Groups of Lie type in cross characteristic}
\label{sec:cross}

In this section we investigate simple endotrivial modules for groups of
Lie type in non-defining characteristic $\ell$. Here our results are not
complete. Throughout this section let $k$ denote a large enough field of
characteristic~$\ell$. We refer to the book of Carter \cite{Ca} for notation
and background.

\subsection{Auxiliary results}

We first prove some general criteria to rule out endotriviality of certain
modules.

Our first observation makes use of Harish-Chandra theory. Let $G$ be a finite
group with a split BN-pair of
characteristic~$p\ne\ell$. Let $Q\le G$ be a parabolic subgroup, with Levi
decomposition $Q=U.L$, so that $U$ is a normal $p$-subgroup of $Q$ with
complement $L$. Let $\tV$ be a $\CC G$-module with character $\chi$, and $V$
the $\ell$-modular reduction of a suitable lattice in $\tV$. The restriction
of $\tV$ to $Q$ decomposes into a direct sum $\tV|_Q=\tV^U\oplus \tV'$, where
no constituent of $\tV'$ has $U$ in its kernel. As customary, we'll
write ${^*}R_L^G(\tV)$ for the $L$-module $\tV^U$ of $U$-fixed points, the
Harish-Chandra restriction of $\tV$. Since all $\ell$-modular constituents of
${^*}R_L^G(\tV)$ have $U$ in their kernel, and none of the $\ell$-modular
constituents of $\tV'$ have, these two summands lie in different $\ell$-blocks
of $G$, and so there is a corresponding decomposition of $V|_Q$ which (by
abuse of notation!) we write $V|_Q={^*}R_L^G(V)\oplus V'$.
Thus we have:

\begin{lem}   \label{lem:HC}
 In the above setting, assume that $V$ is endotrivial. Then exactly one of
 ${^*}R_L^G(V)$, $V'$ is endotrivial, and the other is projective. In
 particular, one has dimension congruent to $\pm1\pmod{|L|_\ell}$, and the
 other has dimension divisible by $|L|_\ell$.
\end{lem}

For finite groups of Lie type, the Harish-Chandra restriction of ordinary
irreducible characters can be computed from information inside relative Weyl
groups. In order to formulate this, we introduce the following setup.
Let $\bG$ be a simple algebraic group of simply connected type over an
algebraic closure of a finite field $\FF_p$, and $F:\bG\rightarrow\bG$ a
Steinberg endomorphism, with finite group of fixed points $G=\bG^F$.
Let $\bG^*$ be a Langlands dual group to $\bG$, with corresponding Steinberg
map also denoted by $F$. We write $q$ for the absolute value of the
eigenvalues of $F$ on the character group of an $F$-stable maximal torus of
$\bG$ (this is a power of $p$, integral unless $G$ is of Ree or Suzuki type).
We then also sometimes write $G=G(q)$. Now Lemma~\ref{lem:HC} leads to the
following:

\begin{lem}   \label{lem:HC e=1}
 Let $G=G(q)$ be a quasi-simple group of Lie type and $\chi\in \Irr(G)$. Assume
 that $V$ is a simple endotrivial $kG$-module, for $k$ a field of
 characteristic~$\ell$ dividing $q-1$, whose Brauer character is the
 restriction of $\chi$ to $\ell$-classes.
 \begin{enumerate}
  \item[\rm(a)] Then $\chi$ lies in the Harish-Chandra series of some (linear)
   character $\theta$ of a maximally split torus $T$ of $G$; let
   $\psi\in\Irr(W_G(T,\theta))$ denote the corresponding irreducible character
   of the relative Weyl group $W_G(T,\theta):=N_G(T,\theta)/T$.
  \item[\rm(b)] We have either $\chi(1)\equiv\psi(1)\equiv\pm1\pmod{|T|_\ell}$
   or $\psi(1)\equiv0\pmod{|T|_\ell}$.
 \end{enumerate}
\end{lem}

\begin{proof}
By Lemma~\ref{lem:deg} we know that $\chi(1)\equiv\pm1\pmod{|G|_\ell}$; in
particular, $\chi$ is of $\ell$-height zero. The first claim is then a
consequence of \cite[Thm.~7.5]{MH0}. \par
Let $B$ be a Borel subgroup containing $T$, with unipotent radical $U$.
By Lemma~\ref{lem:HC}, $V|_B={^*}R_T^G(V)\oplus W$ for some $kB$-module $W$,
and one of ${^*}R_T^G(V)$, $W$ is projective. If ${^*}R_T^G(V)$ is projective
then on the level of characters this means that
${^*}R_T^G(\chi)(1)\equiv0\pmod{|T|_\ell}$, while if ${^*}R_T^G(V)$ is
endotrivial, then ${^*}R_T^G(\chi)(1)\equiv\pm1\pmod{|T|_\ell}$, where the sign
is the same as for $\chi$. \par
Now assume that $\chi$ lies in the Harish-Chandra series of the character
$\theta$ of $T$ and corresponds to $\psi\in\Irr(W_G(T,\theta))$. Then we
have ${^*}R_T^G(\chi)=\psi(1)\theta$, from which we may conclude since $\theta$
is a linear character.
\end{proof}

For our second approach recall Lusztig's partition
$\Irr(G)=\coprod_s\cE(G,s)$ of the irreducible characters of $G$ into disjoint
Lusztig series, indexed by a system of representatives $s$ of the semisimple
conjugacy classes of the dual group $G^*={\bG^*}^F$.

For a prime $\ell$ not dividing $q$ we let $d_\ell(q)$ denote the
multiplicative order of $q$ modulo~$\ell$, respectively $d_2(q):=2$ when
$\ell=2$ and $q\equiv3\pmod4$. We need the following result:

\begin{prop}   \label{prop:d-sylow}
 Let $\chi\in\cE(G,s)$ be the character of a faithful simple endotrivial
 $kG$-module, where $k$ is algebraically closed of
 characteristic~$\ell$. Then $s$ lies in some maximal torus $T\le G^*$
 containing a Sylow $d_\ell(q)$-torus of $G^*$.
\end{prop}

\begin{proof}
Observe that endotrivial modules are of height zero, by Lemma~\ref{lem:deg}.
By \cite[Prop.~7.2]{MH0} any irreducible character of $G$ of $\ell$-height zero
lies in a Lusztig series $\cE(G,s)$ where $C_{G^*}(s)$ contains a Sylow
$\ell$-subgroup of $G^*$. By \cite[Thm.~5.9]{MH0} such elements lie in a
maximal torus of $G^*$ containing a Sylow $d_\ell(q)$-torus $S^*$.
\end{proof}

The following vanishing result (proved for example in \cite[Lem.~3.2]{GMuni})
will be crucial:

\begin{prop}   \label{prop:vanish}
 Let $x\in G$ be semisimple and $\chi\in\cE(G,s)$ with $\chi(x)\ne0$. Then
 there exists a maximal torus $T\le G$ with $x\in T$, and such that
 $T^*\le C_{G^*}(s)$ for a torus $T^*\le G^*$ in duality with $T$.
\end{prop}

\subsection{The case $\ell=2$}

We first show that no new examples arise for $\ell=2$ and thus prove
Theorem~\ref{thm:gen}(a) of the introduction.

Here we use the well-known fact that non-trivial self-dual simple modules in
characteristic~2 have even dimension. This implies in particular:

\begin{cor}   \label{cor:self}
 Let $G$ be a group all of whose elements are real (i.e., conjugate to their
 inverse). Then $G$ cannot have non-trivial simple endotrivial modules in
 characteristic~2. 
\end{cor}

\begin{proof}
Observe that any complex character $\chi$ of $G$ is real valued, since its
complex conjugate satisfies $\bar\chi(g)=\chi(g^{-1})=\chi(g)$ by assumption.
If $V$ is a simple endotrivial module for $G$ in characteristic~2, then by
Theorem~\ref{thm:lift} it is the 2-modular reduction of a simple
$\CC G$-module, which is self-dual since its character is real. But then $V$
is self-dual and thus either trivial or of even dimension; the latter being
excluded by Lemma~\ref{lem:deg}.
\end{proof}

\begin{prop}   \label{prop:unip l=2}
 Let $G$ be a finite simple group of Lie type.
 Let $\chi\ne1$ be a unipotent character of $G$ whose 2-modular reduction
 $\chi^0$ is irreducible. Then $\chi^0$ is not endotrivial.
\end{prop}

\begin{proof}
For groups defined over a field of characteristic~2, this was shown in
Theorem~\ref{thm:defchar} (note that by \cite[\S13]{Ca} the 3-dimensional
characters of $\SL_3(2)$ are not unipotent). So now assume that $G$ is defined
in odd characteristic. Let $\chi$ be a unipotent complex character of $G$.
First assume that $\chi$ is uniquely determined by its multiplicities in the
Deligne--Lusztig characters. Since the latter are rational valued, $\chi$ is
also rational and thus self-dual. Hence the same is true for its 2-modular
reduction. So $\chi^0(1)$ is even by the previous remark, and $\chi^0$ cannot
be endotrivial. \par
Now by the fundamental results of Lusztig if $G$ is a classical group, then
all unipotent characters do have the
above property. For $G$ of exceptional type, the only unipotent characters not
determined by their multiplicities are those for which the associated
eigenvalue of Frobenius is non-real. These are necessarily not in the principal
series, and an easy check shows that all of them have even degree, whence we
may conclude as before.
\end{proof}

\begin{thm}   \label{thm:l=2}
 Let $G$ be a finite quasi-simple group. Then $G$ has a non-trivial simple
 endotrivial $kG$-module $V$ over a field of characteristic~2 if and
 only if one of:
 \begin{enumerate}
  \item[\rm(a)] $G=\PSL_2(q)$ with $7\le q\equiv3\pmod4$ and $\dim(V)=(q-1)/2$;
   or
  \item[\rm(b)] $G=3.\fA_6$ and $\dim(V)\in\{3,9\}$.
 \end{enumerate}
 In particular, Conjecture~\ref{conj:gen} holds for the prime~$2$.  
\end{thm}

\begin{proof}
We go through the various possibilities for $G$ according to the classification.
If $G$ is alternating, the claim is Theorem~\ref{thm:mainAn}, for $G$ sporadic
there are no non-trivial simple endotrivial modules in characteristic~2
by Theorem~\ref{thm:mainspor} below. If $G$ is of Lie type over a field of even
order, the endotrivial simple modules were obtained in
Theorem~\ref{thm:defchar}. If $G$ is an
exceptional group of Lie type in odd characteristic, the claim will follow from
Theorems~\ref{thm:low} and~\ref{thm:exc1} below. Thus we only need to deal with
classical groups of Lie type. Moreover, by Proposition~\ref{prop:unip l=2} we
may assume that $V$ is not the reduction of a unipotent representation. \par

First let $G=\SL_n(q)$ with $q$ odd. By Proposition~\ref{prop:SL2} we
may assume that $n\ge3$. For $n=3$ it follows from the explicitly known
character table \cite{Chv} that among the complex irreducible characters the
only candidates have dimension $q(q^2+q+1)$, and $q\equiv3\pmod4$. But by
\cite[App.]{Hi04} these characters do not remain irreducible modulo~2. 
For $n=4$, the only
non-unipotent complex irreducible characters of odd degree are of
degrees $\Ph3\Ph4/2,q^2\Ph3\Ph4/2$ and lie in the Lusztig series of semisimple
elements with centralizer $\GL_2(q^2).2/(q-1)$ in the dual group. (Here, and
later, we write $\Phi_d$ for the $d$th cyclotomic polynomial evaluated at $q$.)
They satisfy the congruence in Lemma~\ref{lem:deg} when $q\equiv3\pmod4$.
But then there exist elements of order $2\Ph4$ in $G/Z(G)$ (in a torus of
order $(q^4-1)/(q-1)$), while both characters are of $\Ph4$-defect zero.
\par
Now assume that $n\ge5$. By Proposition~\ref{prop:d-sylow}, $V$ is the
reduction of an ordinary representation lying in the Lusztig series of a 
semisimple element $s$ of $G^*=\PGL_n(q)$ contained in the centralizer of a
Sylow $d$-torus, with $d\in\{1,2\}$. These have order $\Ph1^{n-1}$,
$(\Ph1\Ph2)^{\lfloor{(n-1)}/2\rfloor}(\Ph1)^\delta$ respectively, with
$\delta\in\{0,1\}$. Thus, the centralizers of elements in such tori are
contained in centralizers of $\Ph1$- or $\Ph2$-tori, so are (images under the
natural map $\GL_n(q)\rightarrow G^*$) of products of groups $\GL_m(q)$ and
$\GL_m(q^2)$. If $n$ is odd, a maximal torus of type $(n-1)(1)$ of $G^*$
contains regular elements of even order in $[G^*,G^*]$, so $|C_{G^*}(s)|$ must
be divisible by a primitive prime divisor of $q^{n-1}-1$. This is only the
case for $\GL_{n-1}(q)$ and $\GL_{(n-1)/2}(q^2)$. But there also exist
regular elements in maximal tori of type $(n-2)(2)$ and as none of the
candidate groups contains such a torus, we are done by
Proposition~\ref{prop:vanish}. If $n$ is even, we argue similarly that no
centralizer contains maximal tori of types $(n-2)(2)$ and $(n-3)(3)$. 
\par
Next assume that $G=\SU_n(q)$ with $q$ odd, $n\ge3$. When $n=3$, \cite{Chv}
shows that the only candidates are characters of degree $q\Ph6$ when
$q\equiv1\pmod4$, but again these are not irreducible mod~2. For $n=4$, the
only candidates are of degree $\Ph4\Ph6/2,q^2\Ph4\Ph6/2$, with
$q\equiv1\pmod4$. Again, these vanish on elements of order $2\Ph4$ of $G/Z(G)$.
Finally, for $n\ge5$ we may argue as for $\SL_n(q)$, replacing all tori by
their Ennola duals, with order obtained by simply replacing $q$ by $-q$.
\par
If $G$ is of type $B_n,C_n$ or $D_{2n}$, then the Sylow $d$-tori of $G^*$
are maximal tori, hence self-centralizing, and the only elements $s$
centralizing a Sylow 2-subgroup are involutions. By \cite{Chv} no examples
arise for $G=\Sp_4(q)$. Now assume that $G=\Spin_{2n+1}(q)$, $n\ge3$. Recall
that the $G$-conjugacy classes of maximal tori are parametrized by conjugacy
classes in the Weyl group (see e.g. \cite[Prop.~25.1]{MT}), which in turn
are naturally indexed by pairs of partitions of $n$. The maximal tori of $G^*$
of type $((1),(n-1))$ contain
regular elements of even order, thus $C_G(s)$ is of type $C_1C_{n-1}$, or
$\tw2A_n$ when $n$ is even. On the other hand, neither of these contains a
maximal torus of type $(-,(n-2)(2))$, so we may conclude by
Proposition~\ref{prop:vanish}.  Next take $G=\Sp_{2n}(q)$, $n\ge3$. Again
arguing with the maximal torus of type $((1),(n-1))$, we are left with the
possible centralizers of type $B_{n-1},\tw2D_{n-1}B_1,D_n$ and $\tw2D_n$.
Only the one of type $D_n$ contains maximal tori of type $(-,(n-2)(2))$, but
it does not contain a maximal torus of type $((2),(n-2))$.
\par
Next let $G=\Spin_{2n}^+(q)$, $n\ge4$. Arguing with the maximal tori
of type $(-,(n-1)(1))$ we see that only characters in Lusztig series
parametrized by involutions with centralizer of type $\tw2D_{n-1}$, or
$\tw2A_{n-1}$ with $n$ even, matter. But these do not contain tori of type
$(-,(n-2)(2))$. The maximal tori in type $\tw2D_n$ for $n$ odd are just
Ennola dual to those in type $D_n$, and then an analogous argument deals with
this case. Finally, all elements in $\Spin_{4n}^-(q)$, $n\ge2$, are real by
\cite[Thm.~1.2]{TZ}, so no examples arise by Corollary~\ref{cor:self}.
\end{proof}

\subsection{Small rank exceptional groups}
We now consider in more detail the exceptional groups of Lie type, and first
treat the five families of small rank, that is, the Suzuki and Ree groups
$\tw2B_2(2^{2f+1})$, $^2G_2(3^{2f+1})$ and
$\tw2F_4(2^{2f+1})$, and the groups $G_2(q)$ and $\tw3D_4(q)$. For all of these
complete ordinary character tables are available, which makes it relatively
easy to find the candidates for simple endotrivial modules in these cases.

\begin{thm}   \label{thm:low}
 Let $G$ be a covering group of one of the simple groups $\tw2B_2(2^{2f+1})$
 (with $f\ge1$), $^2G_2(3^{2f+1})$ (with $f\ge1$), $G_2(q)$ (with $q\ge3$),
 $\tw3D_4(q)$, or $\tw2F_4(2^{2f+1})$ (with $f\ge1$). Let $\ell\ne p$ denote
 a prime divisor of $|G|$ and $P$ a Sylow $\ell$-subgroup of $G$.
 \begin{enumerate}
  \item[\rm(a)] If there exists a non-trivial simple endotrivial
   $kG$-module then $P$ is cyclic.
  \item[\rm(b)] The simple endotrivial $kG$-modules for primes $\ell$
   such that $P$ is cyclic are precisely as given in Tables~\ref{tab:lowcyc}
   and~\ref{tab:exc1d} (where $\ell |\Phi_d$).
 \end{enumerate}
\end{thm}

\begin{table}[htbp]
 \caption{Simple endotrivial modules for low rank exceptional groups}
  \label{tab:lowcyc}
\[\begin{array}{|r|r|r|c|c|c|}
\hline
 G& d& |X(H)|& |X(H)|/e& \text{block}& \lse(\bB)\\
\hline \hline
 \tw2B_2(q^2)&  1'&    2& 1&   1&  2\\
 \tw2B_2(q^2)&  8'&    4& 1&   1&  3\\
 \tw2B_2(q^2)& 8''&    4& 1&   1&  3\\
\hline \hline
 ^2G_2(q^2)&    1'&    4& 2& 1,2&  2\\
 ^2G_2(q^2)&     4&    6& 1&   1&  5\\
 ^2G_2(q^2)&   12'&    6& 1&   1&  5\\
 ^2G_2(q^2)&  12''&    6& 1&   1&  5\\
\hline \hline
 G_2(q)&         3&    6& 1&   1&  4\\
 G_2(q)&         6&    6& 1&   1&  4\\
\hline \hline
 \tw3D_4(q)&    12&    4& 1&   1&  2\\
\hline \hline
 \tw2F_4(q^2)&  12&    6& 1&   1&  4\\
 \tw2F_4(q^2)& 24'&   12& 1&   1&  8\\
 \tw2F_4(q^2)& 24''&  12& 1&   1&  8\\
\hline
\end{array}\]
Here $\Phi_1'=q^2-1$, $\Phi_8'=q^2+\sqrt{2}q+1$, $\Phi_8''=q^2-\sqrt{2}q+1$,
$\Phi_{12}'=q^2+\sqrt{3}q+1$, $\Phi_{12}''=q^2-\sqrt{3}q+1$,
$\Phi_{24}'=q^4+\sqrt{2}q^3+q^2+\sqrt{2}q+1$,
$\Phi_{24}''=q^4-\sqrt{2}q^3+q^2-\sqrt{2}q+1$.
\end{table}

\begin{table}[htbp]
 \caption{Simple endotrivial modules for exceptional covering groups}
  \label{tab:exc1d}
\[\begin{array}{|r|r|r|c|r|}
\hline
 G& d& |X(H)|& |X(H)|/e& \dim V\\
\hline \hline
 2.\tw2B_2(8)&  5&  8&  2&  56,56\\
 2.\tw2B_2(8)&  7&  4&  2&  64\\
 2.\tw2B_2(8)& 13&  8&  2&  40,40\\
\hline \hline
 3.G_2(3)&      7& 18&  3&  27,27,351,351,729\ (2\times\text{each})\\
 3.G_2(3)&     13& 18&  3&  27,27,378,378,729\ (2\times\text{each})\\
\hline \hline
 2.G_2(4)&      7& 12&  2&  104,104\\
 2.G_2(4)&     13& 12&  2&  12,1260\\
\hline
\end{array}\]
\end{table}

\begin{proof}
We first deal with the case that $G$ itself is simple. First assume that $P$
is cyclic. It turns out that $|X(N_G(P))|/e=1$ in all cases except for
$^2G_2(q^2)$ with $\ell|q^2-1$, so that by Lemma~\ref{lem:X/e} the only
$\ell$-block containing simple endotrivial modules is the principal block.
The position of endotrivial modules on the Brauer tree is then described by
Theorem~\ref{thm:ETleaves}. The Brauer trees for these groups have been
determined by Hi\ss\ \cite{Hi90}. From this, the results in
Table~\ref{tab:lowcyc} follow. For the case of $^2G_2(q^2)$ with $\ell|q^2-1$
we obtain $|X(N_G(P))|/e=2$, and again the corresponding Brauer trees can be
found in \cite{Hi90}.
\par
So now assume that $P$ is not cyclic. Note that all Sylow $\ell$-subgroups
of $\tw2B_2(2^{2f+1})$, for $\ell\ne2$, are cyclic, so there is nothing to
prove for these groups. First assume that $G={}^2G_2(q^2)$, with
$q^2=3^{2f+1}$. The only prime $\ell\ne3$ for which the
Sylow $\ell$-subgroups of $G$ are not cyclic is $\ell=2$. From the known
generic character table (see \Chevie{} \cite{Chv}) it is easy to check that
only one non-trivial complex irreducible character $\chi$, of degree
$q^4-q^2+1$, has the property that its value on
involutions is of absolute value~1. But from the known decomposition matrix
of $G$ in \cite{LMi} it follows that this character does not remain
irreducible modulo~2.
\par
Next assume that $G=G_2(q)$ with $q>2$. The relevant primes $\ell$ in this
case are exactly the prime divisors of $q^2-1$.
From the character tables in \cite{Chv} one sees that all non-trivial
irreducible characters $\chi\in\Irr(G)$ satisfy $|\chi(g)|^2\ne1$ on 2-central
involutions (when $p\ne2$) and on 3-central elements (when $p\ne3$), so we
may assume that $\ell\ge5$ divides exactly one of $q-1$ or $q+1$.
In both cases, the non-trivial characters of degree congruent to
$\pm1\pmod\ell$ are seen to take values either~0 or of absolute value bigger
than~1 on suitable $\ell$-singular elements.
\par
For $G=\tw3D_4(q)$  the relevant primes are the prime divisors of $q^6-1$.
The argument is now completely parallel to the one for $G_2(q)$ above,
using the generic character tables.
\par
Finally assume that $G=\tw2F_4(q^2)$ with $q^2=2^{2f+1}\ge8$. Here, the
relevant primes are the prime divisors of $q^8-1$. We may argue as in the
previous cases, using the generic character table in \cite{Chv}. The only
remaining candidates occur for $\ell=5$. They are characters of degree
$(q^4-1)(q^4-q^2+1)(q^{12}+1)(q^2+\sqrt{2}q+1)$ when $f\equiv1,2\pmod4$, or
$(q^4-1)(q^4-q^2+1)(q^{12}+1)(q^2-\sqrt{2}q+1)$ when $f\equiv0,3\pmod4$. But
according to \cite[Tables~C3 and~C4]{Him}, these characters are reducible
modulo $\ell$.
(They are denoted $\chi_{8,1}$ respectively $\chi_{10,1}$ in loc.~cit.)
\par
The only proper covering groups in our situation are the groups $2.\tw2B_2(8)$,
$3.G_2(3)$ and $2.G_2(4)$. When the Sylow $p$-subgroups are cyclic, we may
conclude by using the criteria in Theorem~\ref{thm:ETleaves} and information
on the Brauer trees. The Sylow $p$-subgroups of the groups in question are
non-cyclic only for $p\le5$ for $2.G_2(4)$, respectively $p\le3$ for
$2.\tw2B_2(8)$ and $3.G_2(3)$. The ordinary character tables are known for all
of these groups and the usual criteria give the claim.
\end{proof}

\subsection{Exceptional groups of large rank}
We now turn to the exceptional groups of rank at least four, for which no
complete generic character tables are available. We obtain an almost complete
picture for unipotent characters:

\begin{prop}   \label{prop:excunip}
 Let $G$ be a finite simple exceptional group of Lie type in characteristic~$p$
 of rank at least~4 and $\ell\ne p$ a prime for which the Sylow $\ell$-subgroups
 of $G$ are non-cyclic. Then the candidates for non-trivial unipotent characters
 with endotrivial $\ell$-modular reduction are given in
 Table~\ref{tab:unipexc}.
\end{prop}

\begin{table}[htbp]
 \caption{Candidates for endotrivial unipotent characters}
  \label{tab:unipexc}
\[\begin{array}{|r|r|r|l|}
\hline
 G& d& \ell& \chi\\
\hline \hline
     F_4&  4&  5&  F_4^{II}[1]\\
     E_6&  4&  5&  D_4,r,\ \phi_{80,7}\\
     E_6&  6& 19&  \phi_{6,25}\\
 \tw2E_6&  4&  5&  \tw2E_6[1],\ \phi_{16,5}\\
     E_8& 10& 31&  \phi_{28,68}\\
\hline
\end{array}\]
Here, the notation for unipotent characters is as in \cite[\S13]{Ca}.
\end{table}

\begin{proof}
First note that we may assume $\ell\ne2$ by Proposition~\ref{prop:unip l=2}.
Also, the Steinberg character is not endotrivial. Indeed,
since we assume that the Sylow $\ell$-subgroups of $G$ are non-cyclic, there
exist $\ell$-elements $g\in G$ with centralizer of positive semisimple rank.
Then $C_G(g)$ contains unipotent elements, and thus there exist
$p$-singular elements in $G$ of order divisible by $\ell$. But the Steinberg
character takes value~0 on all $p$-singular elements, which shows that it
cannot be endotrivial by Corollary~\ref{cor:zero}. \par
The degrees of the unipotent characters of groups of Lie type are known; they
can be found in \cite[\S13]{Ca} or in \cite{Chv}, for example. Let
$d:=d_\ell(q)$ where $G=G(q)$.
The condition that the Sylow $\ell$-subgroups are non-cyclic forces $\Phi_d$
to divide the order polynomial of $G$ at least twice, which restricts the
possible values of $d$. For each type and each such $d$, we use the following
criteria to eliminate candidates:
\begin{itemize}
 \item we have $\chi(1)\not\equiv0\pmod{\Phi_d(q)}$;
 \item if $d=1$ then $\chi(1)$ has to satisfy the congruence in
  Lemma~\ref{lem:HC e=1}(b);
 \item the Harish-Chandra restriction of $\chi$ to proper Levi subgroups $L$
  of $G$ must satisfy the congruences in Lemma~\ref{lem:HC}; and
 \item we have $\chi(1)\equiv\pm1\pmod{|G|_\ell}$.
\end{itemize}
At this stage, usually only very few characters are left, which are then
handled by ad hoc methods. We give some indications for $G=F_4(q)$. For
$d=1,3,6$, no candidates remain. For $d=2$, we are left with the two cuspidal
unipotent characters denoted $F_4^{II}[1]$ and $F_4[-1]$, with
$\ell=11$ respectively $\ell=5$. These are of defect zero for a Zsigmondy
prime divisor (see \cite[Thm.~IX.8.3]{HB82}) $r$ of $\Phi_6$ respectively
$\Phi_3$. Since $B_4(q)\le F_4(q)$ contains maximal tori of orders
$\Ph1\Ph2\Ph3$ and $\Ph1\Ph2\Ph6$ (parametrized by the pairs of partitions
$((3),(1))$ and $((1),(3))$) there exist elements of order $r\ell$ in $G$,
so these characters cannot be
endotrivial by Corollary~\ref{cor:zero}. Note that $q\ne2$ when
$\ell=11|(q+1)$, so there does exist a Zsigmondy prime for $\Phi_6$ in our
situation. For $d=4$, there remain the cuspidal unipotent characters
$F_4[\pm i]$ and $F_4^{II}[1]$ with $\ell=5$. The character $F_4^{II}[1]$
occurs in our list, while the characters $F_4[\pm i]$ are of 3-defect zero,
and $G$ contains elements of order~15.
\par
For $G=E_6(q)$, no candidates remain for $d=1,3$. For $d=2$, all remaining
candidates are of $\Phi_5$- or $\Phi_8$-defect zero, so do not lead to examples
as $G$ contains tori of order $\Ph1\Ph2\Ph5$ and $\Ph1\Ph2\Ph8$. For $d=4$,
there remain the characters denoted $D_4,r$ and $\phi_{80,7}$ with $\ell=5$.
For $d=6$, we are left with $\phi_{6,25}$. When $G=\tw2E_6(q), E_7(q)$ or
$E_8(q)$, a completely similar argument leads to the other three candidates in
Table~\ref{tab:unipexc}.
\end{proof}

\begin{rem}   \label{rem:dont know}
Using decomposition numbers of suitable Hecke algebras one can see that the
unipotent characters $\phi_{80,7}$ for $E_6(q)$ and $\phi_{16,5}$ for
$\tw2E_6(q)$ are reducible modulo primes $\ell$ with $d_\ell(q)=4$, so they
are not endotrivial. On the other hand, the unipotent character $F_4^{II}[1]$
is endotrivial modulo~$\ell=5$ at least for $q=2$ (see Remark~\ref{rem:tors}).
We do not see how to decide
endotriviality for the other cases in Table~\ref{tab:unipexc}.
\end{rem}

We are now ready to complete the proof of Theorem~\ref{thm:gen}(d) of the
introduction:

\begin{thm}   \label{thm:exc1}
 Let $G$ be a quasi-simple exceptional group of Lie type in characteristic $p$,
 and $\ell\ne p$ a prime such that the Sylow $\ell$-subgroups of $G$ have rank at
 least~3. Then $G$ does not have faithful simple endotrivial
 $kG$-modules, where $k$ is a field of characteristic~$\ell$.
\end{thm}

\begin{proof}
By Theorem~\ref{thm:low} we may assume that $G$ is of rank at least~4. If
$G$ is an exceptional covering group $2.F_4(2)$ or $2.\tw2E_6(2)$ then the
condition on $\ell$ forces $\ell=3$. Here the claim follows from the known
ordinary character tables. Thus we have that $G$ is a central factor group of
a finite reductive group of simply connected type. Let $d=d_\ell(q)$. Then
the cases to consider are:
$d=1,2$ for $F_4(q)$, $d=1,2,3$ for $E_6(q)$, $d=1,2,6$ for $\tw2E_6(q)$,
$d=1,2,3,6$ for $E_7(q)$ and $d=1,2,3,4,6$ for $E_8(q)$. \par
Let $V$ be a faithful simple endotrivial $kG$-module, the
$\ell$-modular reduction of a $\CC G$-module with character $\chi$. Since all
candidates in Proposition~\ref{prop:excunip} occur for $\ell$-rank~2, we know
that $\chi$ is not unipotent, so $\chi$ lies in some Lusztig series $\cE(G,s)$
with $s\ne1$. Furthermore, $s$ must lie in the centralizer of a Sylow
$d$-torus of $G^*$ by Proposition~\ref{prop:d-sylow}. In all cases considered
here, $d$ is a Springer regular number for the Weyl group of $G$
(see e.g.~\cite[p.~260]{BM92}), so the centralizer of
a Sylow $d$-torus of $G^*$ is a maximal torus $T^*$ of $G^*$. Moreover, a
Sylow $d$-torus of $G^*$ is already a maximal torus, whence equals $T^*$,
unless $G=E_6(q)$ and $d=2$, $G=\tw2E_6(q)$ and $d=1$, or $G=E_7(q)$ and
$d=3,6$. In Table~\ref{tab:toriexc} in each relevant case we list two or
three maximal tori of $G^*$ (in Carter's notation, see also \cite{Chv}).
We have omitted $\tw2E_6(q)$ since the relevant tori in that group are
obtained from those in $E_6(q)$ by formally replacing $q$ by $-q$ (see
\cite[\S3B]{BM92}).

\begin{table}[htbp]
 \caption{Large $\ell$-rank in exceptional groups}
  \label{tab:toriexc}
\[\begin{array}{|crl||crl|}
\hline
     G& d& \text{tori}&     G& d& \text{tori}\cr
\hline
 F_4&  1& A_2+\tilde A_1,\tilde A_2+A_1,A_3&  E_6&  1& D_5, A_4+A_1\\
    &  2& B_3,C_3,A_3&                           &  2& D_5,A_4+A_1\\
 E_8&  1& E_7,A_7''&                             &  3& A_2+2A_1,E_6\\
    &  2& A_8,A_7''&                          E_7&  1& E_6(a_1),A_6\\
    &  3& A_8,D_7&                               &  2& E_7,E_7(a_1)\\
    &  4& D_7(a_1),A_7,A_7''&                    &  3& A_4+A_2,E_6\\
    &  6& E_8(a_4),D_7&                          &  6& E_7(a_2),E_7(a_3)\\
\hline
\end{array}\]
\end{table}

Now in all cases, except when $G=E_8(q)$ and $d=4$, the only element $s\in T^*$
such that $C_{G^*}(s)$ contains conjugates of the two (or three) tori listed
in the table, is $s=1$. This is easily checked by using that maximal tori are
parametrized by (possibly twisted) conjugacy classes in the Weyl group; thus
one just has to verify that the Weyl coset of any centralizer $C_{G^*}(s)$
does not contain representatives from all two or three conjugacy classes.
Since all listed tori do contain regular elements, Proposition~\ref{prop:vanish}
implies that $\chi\in\cE(G,s)$ cannot be endotrivial when $s\ne1$ and
$(G,d)\ne(E_8(q),4)$.
\par
When $G=E_8(q)$ and $d=4$, there is an isolated element $s\in T^*$ of order~2
whose centralizer $D_8(q)$ contains all listed maximal tori. It remains to
show that $\cE(G,s)$ for this element~$s$ does not contain characters of
endotrivial modules. As $d=4$ we have $\ell\ge5$. Now the approach given for
unipotent characters in the proof of Proposition~\ref{prop:excunip} using
congruences and Harish-Chandra restriction rules out all characters in this
series.
\end{proof}

\section{Covering groups of sporadic simple groups} \label{sec:spor}

\begin{thm}   \label{thm:mainspor}
 Let $G$ be a quasi-simple group such that $G/Z(G)$ is sporadic simple. Let
 $V$ be a faithful simple endotrivial $kG$-module, where $k$ is
 algebraically closed of characteristic $p$, with $p$ dividing $|G|$. Let
 $P$ be a Sylow $p$-subgroup of $G$. Then one of the following holds:
 \begin{enumerate}
  \item[\rm(1)] $|P|=p$ and $V$ lies in a $p$-block $\bB$ of $kG$ as indicated
   in Table~\ref{tab:spd1}; or
  \item[\rm(2)] $(G,P,\dim V)$ are as in Table~\ref{tab:spor2}.
 \end{enumerate}
 Conversely, all modules listed in Table~\ref{tab:spor2} are endotrivial
 except possibly for those (of dimension at least 5824) marked by a "?" in the
 last column.
\end{thm}

Table~\ref{tab:spd1} also gives the number $\lse(\bB)$ of simple endotrivial
modules in the block $\bB$, except for one block of $2.B$ in characteristic~47
and five blocks of $M$ in various characteristics, where the Brauer trees are
not known completely. The numbering of the blocks is as given by Hi\ss\ and
Lux in \cite{HL}. 

\begin{proof}
First assume $p^2$ divides $|G|$. Let $\chi$ be the ordinary irreducible
character of $G$ belonging to the lift of $V$ from Theorem~\ref{thm:lift}.
Using Lemma~\ref{lem:deg} and Corollary~\ref{cor:zero} and the known ordinary
character tables of the quasi-simple sporadic groups (see \cite{Atl}) we obtain
the list of candidates for $\chi$.  \par
Using the character tables given in \cite{MAtl, GAP} we may discard the
characters whose reduction modulo~$p$ is not irreducible. Thus
Table~\ref{tab:spor2} lists those characters whose restriction modulo~$p$ is
irreducible, or, in the case of $J_4$, $Fi_{24}'$, $B$ and $M$, those for
which the question of irreducibility modulo $p$ is still open (these cases
are indicated by a question mark in the corresponding line). The sheer size
of these modules makes it impossible to do any direct computations.

\begin{table}[htbp]
 \caption{Candidate characters in sporadic groups}
  \label{tab:spor2}
\vskip -.5pc
\[\begin{array}{|c|l|l||c|l|l|c|}
\hline
     G& P& \chi(1)&  G& P& \chi(1)& \cr
\hline
   M_{11}& 3^2&  10,10,10&        Fi_{22}& 5^2&  1001&\cr
   M_{22}& 3^2&  55&            2.Fi_{22}& 5^2&  5824\ (4\times)& ?\cr
 2.M_{22}& 3^2&  10,10, 154,154& 3.Fi_{22}& 5^2&  351,351, 12474\ (4\times)&?\cr
   M_{23}& 3^2&  253&           6.Fi_{22}& 5^2&  61776\ (4\times)& ?\cr
       HS& 3^2&  154,154,154&          Th& 7^2&  27000,27000& ?\cr
    3.McL& 5_+^{1+2}&  126,126,126,126& Fi_{23}& 5^2&  111826&\cr
       He& 5^2&  51,51&   J_4& 11_+^{1+2}& 887778,887778, 394765284& ?\cr 
       Ru& 3_+^{1+2}&  406&        Fi_{24}'& 5^2&  74887473024& ?\cr  
     2.Ru& 3_+^{1+2}&  28,28&  B& 7^2& 9287037474, 775438738408125& ?\cr 
      Suz& 5^2&  1001&         M& 11^2&  7226910362631220625000& ?\cr 
     3.ON& 7_+^{1+2}&  342,342,342,342& \tw2F_4(2)'&3_+^{1+2}&  26,26&\cr
         &  &                 &  \tw2F_4(2)'&5^2&  26,26, 351, 351&\cr
\hline
\end{array}\]
\end{table}

The modules for $M_{11},M_{22},2.M_{22},M_{23}$ with $p=3$ are indeed
endotrivial by \cite[\S2.3]{Schu}. For the cases
$$\begin{aligned}
(G,p,\chi(1))\in& \{(2.M_{22},3,10),(3.McL,5,126),(2.Ru,3,28),(3.ON,7,342), \\
                & \,\,(\tw2F_4(2)',3,26),(\tw2F_4(2)',5,26)\}
\end{aligned}$$
it can be seen from the character table that the tensor product
$\chi\otimes \chi^*$ has one trivial constituent and one constituent of
defect zero (see also \cite[Table~7.1]{MaT}), so clearly the corresponding
$kG$-module $V$ is endotrivial. For the following configurations
the restriction of the corresponding ordinary character to a subgroup $H$
containing a Sylow $p$-subgroup of $G$ has a unique trivial constituent, and
all other constituents are of defect zero:
$$\begin{aligned}
(G,p,\chi(1),H)\in&\{(M_{23},3,253,M_{22}),(He,5,51,\PSp_4(4).2),
        (Ru,3,406,\tw2F_4(2)'),\\
        &(Suz,5,1001,G_2(4)), (Fi_{22},5,1001,\OO_8^+(2).\fS_3),(3.Fi_{22},5,351,\OO_8^+(2)),\\
        &(Fi_{23},5,111826,\OO_8^+(3).\fS_3)\}
\end{aligned}$$
so $\chi$ is the character of an endotrivial module by Lemma~\ref{lem:sub}.
Computations with \MAGMA{} \cite{MAG} show that the modules
$(G,p,\chi(1))\in\{(HS,3,154),(\tw2F_4(2)',5,351)\}$ are endotrivial.\par
Finally, we consider the cases with cyclic Sylow $p$-subgroup. The results
are collected in Table~\ref{tab:spd1}, whose entries are obtained as follows.
Let $H:=N_{G}(Z)$ as in Section~\ref{sec:cycl}, $|X(H)|=|H/[H,H]|_{p'}$
and let $e$ denote the inertial index of the principal block. Then $|X(H)|$ and
$|X(H)|/e$ are the bounds for the number of simple endotrivial $kG$-modules
given by Corollary~\ref{cor:bounds}. We then determine the simple endotrivial
$kG$-modules using the descriptions of the Brauer trees in \cite{HL},
resp.~\cite{Nae02} for $M$ in characteristic~29, as well as
Lemma~\ref{lem:X/e}, Lemma~\ref{lem:2e},  Theorem~\ref{thm:ETleaves},
Lemma~\ref{lem:deg} and Lemma~\ref{cor:zero}.
\end{proof}

\begin{rem}   \label{rem:tors}
If $V$ is the reduction modulo $p$ of a $\CC G$-module with character $\chi$
such that
$$\begin{aligned}
 (G,p,\chi(1),H)\in&\{(M_{23},3,253,M_{22}),(He,5,51,\PSp_4(4).2),
        (Ru,3,406,\tw2F_4(2)'),\\
        &(Suz,5,1001,G_2(4)), (Fi_{22},5,1001,\OO_8^+(2).\fS_3),(3.Fi_{22},5,351,\OO_8^+(2)),\\
        &(Fi_{23},5,111826,\OO_8^+(3).\fS_3)\}
\end{aligned}$$
as in the proof of Theorem~\ref{thm:mainspor}, then $V$ is not only
endotrivial, but also a trivial source $kG$-module because
$V|_{H}\cong k\oplus \text{(projective)}$. Therefore the class of $V$ in the
group $T(G)$ of endotrivial modules is a torsion element. Another such example
is given by the unipotent character $F_4^{II}[1]$ of $F_4(2)$ of degree~1326
from
Proposition~\ref{prop:excunip} (with the subgroup $H=\OO_8^+(2).3.2$). \par
Thus we have obtained a list of non-trivial torsion elements of $T(G)$ unknown
in the literature so far. This is of particular interest because a main problem
remaining in the classification of endotrivial modules is the determination
of the torsion subgroup of $T(G)$.
\end{rem}

\begin{rem}
Endotriviality is in general not preserved by Morita equivalences. Let us point
out the following example: the Janko group $J_{1}$ has four 3-blocks with
isomorphic Brauer trees, that is, isomorphic as pointed graphs equipped with
a planar embedding. Hence the four blocks are Morita equivalent
(see \cite[pp.~69--70]{HL}). However, Table~\ref{tab:spd1}
shows that only two of these blocks contain simple endotrivial modules.
\end{rem}

\begin{table}[htbp]
 \caption{Blocks of sporadic groups with cyclic defect containing simple endotrivial modules}
  \label{tab:spd1}
\[\begin{array}{|r|r|r|c|c|c|}
\hline
 G& p& |X(H)|& |X(H)|/e& \text{block } \bB&  \lse(\bB)\\
\hline \hline
\tw2F_4(2)'& 13& 6& 1&  1&    4\\
\hline \hline
M_{11}&    5&  4& 1& 1&      4\\ \cline{2-6}
M_{11}&   11&  5& 1& 1&      3\\
\hline \hline
M_{12}&    5&  8& 2& 1&       2\\
M_{12}&    5&   &  & 2&       4\\ \cline{2-6}
M_{12}&   11&  5& 1& 1&       1\\
2.M_{12}& 11& 10& 2& 2&       4\\
\hline \hline
J_{1}&     3&  4&  2&  1,4&    2\\   \cline{2-6}
J_{1}&     5&  4&  2&  1&      2\\
J_{1}&     5&   &   &  3&      1\\    \cline{2-6}
J_{1}&     7&  6&  1&  1&      2\\     \cline{2-6}
J_{1}&    11& 10&  1&  1&     2\\     \cline{2-6}
J_{1}&    19&  6&  1&  1&      2\\    \cline{2-6}
\hline \hline
M_{22}&    5&  4&  1&  1&     2\\
2.M_{22}&  5&  8&  2&  2&     2\\
3.M_{22}&  5& 12&  3&  3,4&   2\\
4.M_{22}&  5& 16&  4&  5,6&   2\\
6.M_{22}&  5& 24&  6&  7,8&   4\\
12.M_{22}& 5& 48& 12&  9,10,11,12&  3 \\  \cline{2-6}
M_{22}&    7&  3&  1&  1&      1\\
2.M_{22}&  7&  6&  2&  2&      1\\
3.M_{22}&  7&  9&  3&  3,4&    1\\
4.M_{22}&  7& 12&  4&  5,6&    1\\
6.M_{22}&  7& 18&  6&  7,8&    1\\
12.M_{22}& 7& 36& 12&  9,10,11,12&  2\\  \cline{2-6}
M_{22}&   11&  5&  1&  1&      3\\
2.M_{22}& 11& 10&  2&  2&      3\\
3.M_{22}& 11& 15&  3&  3,4&    4\\
4.M_{22}& 11& 20&  4&  5,6&    4\\
6.M_{22}& 11& 30&  6&  7,8&    2\\
12.M_{22}& 11& 60& 12& 9,10&   2\\
12.M_{22}& 11&  &   &  11,12&  4\\
\hline \hline
J_{2}&    7&   6&  1&  1&      2\\
2.J_{2}&  7&  12&  2&  2&      4 \\
\hline \hline
M_{23}&   5&  4& 1&  1&       4\\   \cline{2-6}
M_{23}&   7&  8& 2&  1,2&     1\\   \cline{2-6}
M_{23}&  11&  5& 1&  1&       3\\   \cline{2-6}
M_{23}&  23& 11& 1&  1&      5\\
\hline
\end{array}\]
\end{table}

\begin{table}[htbp]
\vskip -1pt
\[\begin{array}{|r|r|r|c|c|c|}
\hline
 G& p& |X(H)|& |X(H)|/e& \text{block } \bB&  \lse(\bB)\\
\hline \hline
HS&       7&  6& 1&  1&       2\\
2.HS&     7& 12& 2&  2&       2\\   \cline{2-6}
HS&      11&  5& 1&  1&       1\\
2.HS&    11& 10& 2&  2&       2\\
\hline \hline
J_{3}&    5&  4& 2&  1,3&    1\\
3.J_{3}&  5& 12& 6&  4,5,6,7&  1\\ \cline{2-6}
J_{3}&   17&  8& 1&  1&      2 \\
3.J_{3}& 17& 24& 3& 2,3&     6\\  \cline{2-6}
J_{3}&   19&  9& 1&  1&      1 \\
3.J_{3}& 19& 27& 3& 2,3&     6 \\
\hline \hline
M_{24}&   5&  4& 1&  1&      2\\   \cline{2-6}
M_{24}&   7&  6& 2&  1,3&    1\\   \cline{2-6}
M_{24}&  11& 10& 1&  1&     6\\   \cline{2-6}
M_{24}&  23& 11& 1&  1&      7\\   \cline{2-6}
\hline \hline
McL&      7&  6& 2&  1,2&    2 \\
3.McL&    7& 18& 6& 3,4&     1\\
3.McL&    7&   &  & 5,6&     2\\  \cline{2-6}
McL&     11&  5& 1&  1&      1\\
3.McL&   11& 15& 3& 2,3&     1\\
\hline \hline
He&      17&  8& 1&  1&      1\\
\hline \hline
Ru&       7&  6& 1&  1&       2\\  \cline{2-6}
Ru&      13& 12& 1&  1&     4\\  \cline{2-6}
Ru&      29& 14& 1&  1&      4\\
2.Ru&    29& 28& 2&  3&      9 \\
\hline \hline
Suz&      7&  6& 1&  1&       2\\
3.Suz&    7&   &  &  8,9&     2\\     \cline{2-6}
Suz&     11& 10& 1&  1&      2\\
2.Suz&   11& 20& 2&  2&      4\\
3.Suz&   11& 30& 3&  3,4&     6\\
6.Suz&   11& 60& 6&  5,6&     6\\ \cline{2-6}
Suz&     13&  6& 1&  1&        1\\
2.Suz&   13& 12& 2&  2&        1\\
3.Suz&   13& 18& 3&  3,4&      1\\
6.Suz&   13& 36& 6&  5,6&      3\\
\hline \hline
ON&       5&  8& 2&  1&      2 \\
ON&       5&   &  &  2&       4\\ \cline{2-6}
ON&      11& 10& 1&  1&     2\\
3.ON&    11& 30& 3&  2,3&   4\\ \cline{2-6}
ON&      19&  6& 1&  1&      2\\
3.ON&    19& 18& 3&  2,3&   3\\ \cline{2-6}
ON&      31& 15& 1&  1&     1\\
3.ON&    31& 45& 3&  2,3&   8\\
 \hline
\end{array}\]
\end{table}

\begin{table}[htbp]
\vskip -1pt
\[\begin{array}{|r|r|r|c|c|c|}
\hline
 G& p& |X(H)|& |X(H)|/e& \text{block } \bB&  \lse(\bB)\cr
\hline \hline
Co_3&     7& 12& 2&  1&      4\\
Co_3&     7&   &  &  3&      2\\  \cline{2-6}
Co_3&    11& 10& 2&  1,2&    1\\  \cline{2-6}
Co_3&    23& 11& 1&  1&     5\\
\hline \hline
Co_2&     7& 12& 2&  1,3&     2\\ \cline{2-6}
Co_2&    11& 10& 1&  1&      4\\ \cline{2-6}
Co_2&    23& 11& 1&  1&      5\\
\hline \hline
Fi_{22}&    7& 12& 2&  1,3&     2\\
2.Fi_{22}&  7& 24& 4&  5,6&     2\\
3.Fi_{22}&  7& 36& 6&  7,8,9,10&  2\\
6.Fi_{22}&  7& 72& 12& 13,14,15,16&  2\\  \cline{2-6}
Fi_{22}&   11& 10& 2&  1,2&     1\\
2.Fi_{22}& 11& 20& 4&  3,4&     1\\
3.Fi_{22}& 11& 30& 6&  5,6,7,8&   2\\
6.Fi_{22}& 11& 60& 12& 9,10,11,12&  2 \\ \cline{2-6}
Fi_{22}&   13&  6& 1&  1&       1\\
2.Fi_{22}& 13& 12& 2&  2&       4\\
3.Fi_{22}& 13& 18& 3&  3,4&     1\\
6.Fi_{22}& 13& 36& 6&  5,6&     2\\
\hline \hline
HN&         7&  6& 1&  1&      2\\ \cline{2-6}
HN&        11& 20& 2&  1&     2\\
HN&        11&   &  &  2&     4\\ \cline{2-6}
HN&        19&  9& 1&  1&      1\\
\hline \hline
Ly&         7&  6& 1&  1&        2\\ \cline{2-6}
Ly&        11& 10& 2&  1,2&      1\\ \cline{2-6}
Ly&        31&  6& 1&  1&        4\\ \cline{2-6}
Ly&        37& 18& 1&  1&       6\\ \cline{2-6}
Ly&        67& 22& 1&  1&       6\\
\hline \hline
Th&        13& 12& 1&  1&      8\\   \cline{2-6}
Th&        19& 18& 1&  1&       10\\   \cline{2-6}
Th&        31& 15& 1&  1&       11\\
\hline \hline
Fi_{23}&     7& 12& 2&  1,4&      2\\ \cline{2-6}
Fi_{23}&    11& 20& 2&  1&   2\\
Fi_{23}&    11&   &  &  3&   4\\ \cline{2-6}
Fi_{23}&    13& 12& 2&  1,3&    2\\ \cline{2-6}
Fi_{23}&    17& 16& 1&  1&   4\\ \cline{2-6}
Fi_{23}&    23& 11& 1&  1&   1\\
\hline \hline
Co_1&       11& 20& 2&  1&   2\\
Co_1&       11&   &  &  2&   4\\ \cline{2-6}
Co_1&       13& 12& 1&  1&   4\\ \cline{2-6}
Co_1&       23& 11& 1&  1&   1\\
2.Co_1&     23& 22& 2&  2&   5\\
 \hline
\end{array}\]
\end{table}

\begin{table}[htbp]
\[\begin{array}{|r|r|r|c|c|c|}
\hline
 G& p& |X(H)|& |X(H)|/e& \text{block } \bB& \lse(\bB)\cr
\hline \hline
J_4&         5&  4& 1&  1&         2 \\ \cline{2-6}
J_4&         7&  6& 2&  1&         1 \\
J_4&         7&   &  &  2&         1 \\ \cline{2-6}
J_4&        23& 22& 1&  1&       10 \\ \cline{2-6}
J_4&        29& 28& 1&  1&       10 \\ \cline{2-6}
J_4&        31& 10& 1&  1&        2 \\ \cline{2-6}
J_4&        37& 12& 1&  1&        6 \\ \cline{2-6}
J_4&        43& 14& 1&  1&        8 \\ \cline{2-6}
\hline \hline
Fi_{24}'&   11& 10& 1&  1&        2 \\
3.F_{24}'&  11& 30& 3&  8,9&      4 \\   \cline{2-6}
Fi_{24}'&   13& 24& 2&  1,4&      4 \\  \cline{2-6}
Fi_{24}'&   17& 16& 1&  1&        2 \\
3.Fi_{24}'& 17& 48& 3&  2,3&      6 \\   \cline{2-6}
Fi_{24}'&   23& 11& 1&  1&        1  \\
3.Fi_{24}'& 23& 33& 3&  2,3&      5 \\  \cline{2-6}
Fi_{24}'&   29& 14& 1&  1&        3  \\
3.Fi_{24}'& 29& 42& 3&  2,3&      5 \\
\hline \hline
B&          11& 20& 2&  1,7&      2  \\   \cline{2-6}
B&          13& 24& 2&  1,5&      2  \\   \cline{2-6}
B&          17& 32& 2&  1&        6  \\
B&          17&   &  &  3&        4  \\   \cline{2-6}
B&          19& 36& 2&  1&       4 \\
B&          19&   &  &  2&        6 \\   \cline{2-6}
B&          23& 22& 2&  1,2&      1   \\
2.B&        23& 44& 4&  3&        5   \\
2.B&        23&   &  &  4&        7   \\  \cline{2-6}
B&          31& 15& 1&  1&        1 \\
2.B&        31& 30& 2&  2&        9 \\  \cline{2-6}
B&          47& 23& 1&  1&        5 \\
2.B&        47& 46& 2&  2&       \ge3 \\
\hline \hline
M&          17& 16& 1&  1&        6  \\ \cline{2-6}
M&          19& 18& 1&  1&        6  \\ \cline{2-6}
M&          23& 22& 2&  1&        5  \\
M&          23&   &  &  5&        1 \\ \cline{2-6}
M&          29& 28& 1&  1&        \ge10 \\ \cline{2-6}
M&          31& 30& 2&  1&        3  \\
M&          31&   &  &  3&        5  \\   \cline{2-6}
M&          41& 40& 1&  1&      \ge8 \\   \cline{2-6}
M&          47& 46& 2&  1&      \ge3 \\
M&          47&   &  &  2&        9\\   \cline{2-6}
M&          59& 29& 1&  1&      \ge15 \\  \cline{2-6}
M&          71& 35& 1&  1&      \ge7 \\
\hline
\end{array}\]
\end{table}


\end{document}